\documentclass[a4paper,11pt,reqno]{amsart}
\usepackage[foot]{amsaddr}

\usepackage[english]{babel}
\usepackage[utf8]{inputenc}
\usepackage{longtable}
\usepackage{mathtools}
\usepackage{fixmath}
\usepackage{amssymb}
\usepackage{multirow}
\usepackage{alltt}
\usepackage{graphicx}
\usepackage{amsmath}
\usepackage{bbm}
\usepackage{soul}
\usepackage{siunitx}
\usepackage[T1]{fontenc}
\usepackage{textcomp}
\usepackage{fullpage}
\usepackage{comment}
\usepackage[usenames,dvipsnames]{xcolor}
\definecolor{darkblue}{RGB}{0,0,170}
\definecolor{darkerblue}{RGB}{0,20,120}
\definecolor{brickred}{RGB}{200,0,0}
\definecolor{tempcolor}{RGB}{200,0,0}
\usepackage{soul}
\usepackage{booktabs}

\usepackage[pdftitle={Sparse Inverse Problems}, hyperindex]{hyperref}
\usepackage[margin=0.9in, right=0.9in, left=0.9in]{geometry}
\hypersetup{colorlinks=true, linkcolor=darkblue, citecolor=brickred}
\makeatletter
\def\@seccntformat#1{%
	\protect\textup{%
		\protect\@secnumfont
		\expandafter\protect\csname format#1\endcsname 
		\csname the#1\endcsname
		\protect\@secnumpunct
	}%
}
\makeatother

\makeatletter
\def\th@plain{%
  \thm@notefont{}
  \slshape 
}
\def\th@definition{%
  \thm@notefont{}
  \normalfont 
}
\makeatother

\theoremstyle{plain}
\newtheorem{definition}{Definition}[section]
\newtheorem{example}[definition]{Example}
\newtheorem{theorem}[definition]{Theorem}
\newtheorem{proposition}[definition]{Proposition}
\newtheorem{lemma}[definition]{Lemma}
\newtheorem{remark}[definition]{Remark}

\addto\captionsenglish{}

\newcommand{\R}{\mathbb{R}}

\newcommand{\M}{\mathcal{M}}

\renewcommand{\P}{\mathbb{P}}
\newcommand{\E}{\mathbb{E}}
\newcommand{\eps}{\varepsilon}

\newcommand{\parameter}{k,\mu^{\dagger},\normI}

\DeclarePairedDelimiter{\norm}{\lVert}{\rVert}
\DeclarePairedDelimiter{\abs}{\lvert}{\rvert}

\DeclarePairedDelimiter{\inner}{(}{)}
\newcommand{\de}{\mathop{}\!\mathrm{d}}

\DeclareMathOperator{\supp}{supp}
\DeclareMathOperator{\id}{Id}
\DeclareMathOperator{\tr}{tr}

\DeclareMathOperator{\PC}{PC}

\DeclareMathOperator{\sign}{sign}

\DeclareMathOperator{\diag}{diag}

\DeclareMathOperator{\MSE}{MSE}
\DeclareMathOperator{\KR}{KR}
\DeclareMathOperator{\TV}{TV}
\DeclareMathOperator{\HK}{HK}

\DeclareMathOperator*{\argmin}{arg\,min}

\newcommand{\miu}{\mu}
\newcommand{\bld}[1]{\boldsymbol{#1}}

\newcommand{\ke}{k}
\newcommand{\Ke}{K}
\renewcommand{\epsilon}{\varepsilon}

\newcommand{\normI}{\norm{\mathcal{I}_0^{-1}}_{W_{\dagger}^{-1} \to W_{\dagger}}}
\newcommand{\Prec}{\Sigma_0^{-1}}
\newcommand{\ptot}{p}
\newcommand{\zd}{z^d}

\newcommand{\vectorsgn}{(\bld{\rho};\bld{0})}
\usepackage{enumitem}
\setenumerate{nolistsep}

\usepackage{todonotes}

\newcounter{constantC}
\newcommand{\newconstant}[1]{\refstepcounter{constantC}\label{#1}}
\newcommand{\useconstant}[1]{C_{\ref{#1}}}
\newcommand{\defconstant}[1]{ \newconstant{C_#1}\expandafter\newcommand\csname C#1\endcsname{\useconstant{C_#1}}}  %

\newcounter{constantc}
\newcommand{\newconstantc}[1]{\refstepcounter{constantc} \label{#1}}
\newcommand{\useconstantc}[1]{c_{\ref{#1}}}
\newcommand{\defconstantc}[1]{ \newconstantc{c_#1}\expandafter\newcommand\csname c#1\endcsname{\useconstantc{c_#1}}}

\numberwithin{equation}{section}

\usepackage{lmodern}
\newcommand{\ca}{c_1}
\newcommand{\cc}{c_2}
\newcommand{\cd}{c_3}
\newcommand{\ce}{c_4}

\newcommand{\Ca}{C_1}
\newcommand{\Cb}{C_2}
\newcommand{\Cd}{C_3}
\newcommand{\Ce}{C_4}

\begin{document}

\title[Optimal design]{
Towards optimal sensor placement for inverse problems in spaces of measures
}
\author[P.-T. Huynh]{Phuoc-Truong Huynh}
\address{Institut f\"{u}r Mathematik, Alpen-Adria-Universit\"{a}t Klagenfurt,  9020 Klagenfurt, Austria.}
\email{phuoc.huynh (at) aau.at}

\author[K. Pieper]{Konstantin Pieper}
\address{Computer Science and Mathematics Division, Oak Ridge National Laboratory, Oak Ridge, TN 37831, USA.}
\email{pieperk (at) ornl.gov}

\author[D. Walter]{Daniel Walter}
\address{Institut f\"{u}r Mathematik, Humboldt-Universit\"{a}t zu Berlin, 10117 Berlin, Germany.}
\email{daniel.walter (at) hu-berlin.de}


\begin{abstract} 

The objective of this work is to quantify the reconstruction error in sparse inverse problems with measures and stochastic noise, motivated by optimal sensor placement.
To be useful in this context, the error quantities must be explicit in the sensor configuration and robust with respect to the source, yet relatively easy to compute in practice, compared to a direct evaluation of the error by a large number of samples.

In particular, we consider the identification of a measure consisting of an unknown linear combination of point sources from a finite number of measurements contaminated by Gaussian noise.
The statistical framework for recovery relies on two main ingredients: first, a convex but non-smooth variational Tikhonov point estimator over the space of Radon measures and, second, a suitable mean-squared error based on its Hellinger-Kantorovich distance to the ground truth. 
To quantify the error, we employ a non-degenerate source condition as well as careful linearization arguments to derive a computable upper bound.
This leads to asymptotically sharp error estimates in expectation that are explicit in the sensor configuration. Thus they can be used to estimate the expected reconstruction error for a given sensor configuration and guide the placement of sensors in sparse inverse problems.

\bigskip
	
	
\bigskip
	

\end{abstract}

\maketitle

\section{Introduction}
The identification of an unknown signal~$\mu^\dagger$ comprising finitely many point sources lies at the heart of challenging applications such as acoustic inversion~\cite{pieper_tang_trautmann_walter_2020,Trautmannseismic}, microscopy~\cite{McCutchen:67,denoyelle_duval_peyre_2020}, astronomy~\cite{astronomy}, 
{low-rank tensor decomposition~\cite{li2015overcomplete}, linear system identification~\cite{schiebinger}}, as well as initial value identification~\cite{vexlerinitial,Casas_2019,vexler}.
{
Moreover, the recovery of an unknown function by one-hidden-layer neural networks~\cite{bach,chizat,NonconvexNN_2022} is intrinsically linked to this task.
In all of these contexts the problem is to identify an unknown linear combination (superposition) of functions indexed by a nonlinear parameter from a finite number of measurements.
Motivated by inverse point source location tasks we will refer to the linear parameters as \emph{amplitudes} and nonlinear parameters as \emph{locations}.
Moreover, we will assume that measurements are associated to certain spatial locations, motivated by point-wise measurements of physical quantities.%
}
Denoting by~$\Omega_s \subset \R^d$ and $\Omega_o \subset \R^{d_o}$,~$d,d_o \geq 1$, compact sets of possible source locations and measurement points, a common mathematical framework for the recovery of the locations~$y^\dagger_n \in \Omega_s $ and amplitudes~$q^\dagger_n$ of its~$N^\dagger_s$ individual point sources can be given by equations of the form
\begin{equation} \label{eq:inverse}
z^d_j(\varepsilon) 
= \sum^{N^\dagger_s}_{n=1} q^\dagger_n k(x_j,y^\dagger_n)+ \varepsilon_j \quad \text{for}~j=1,\dots,N_o;
\end{equation}
Here,~$k \in \mathcal{C}(\Omega_o \times \Omega_s)$ denotes a sufficiently smooth given integral kernel (resulting from the modeling of the physical process and the properties of the sensors), and~$x_j \in \Omega_o$ denote measurement locations.
Moreover, \(\varepsilon_j\) is a measurement error for each sensor that, for the purposes of this paper is thought of as a random perturbation stemming from measurement noise.
This type of \textit{ill-posed inverse problem} is challenging for a variety of reasons. First and foremost, we neither assume knowledge of the amplitudes and positions of the sources nor of their number. This adds an additional combinatorial component to the generally nonlinear nonconvex problem. Second, inference on~$\mu^\dagger$ is only possible through a finite number of indirect measurements~$z^d$. Additional challenges are given by the appearance of unobservable measurement noise~$\varepsilon$ in the problem.

To alleviate some of these difficulties we identify~$\mu^\dagger$ with a finite linear combination of Dirac measures
\begin{align} \label{def:sparse}
\mu^\dagger= \sum^{N^\dagger_s}_{n=1} q^\dagger_n \delta_{y^\dagger_n} 
\quad\text{and }
z^d_j(\varepsilon) = \int_{\Omega_s} k(x_j,y)\de\mu^{\dagger}(y) + \varepsilon_j.
\end{align}
Subsequently, we try to recover~$\mu^\dagger$ by the stable solution of the linear, ill-posed, operator equation:
\begin{align*}
\text{Find}~\mu \in \M(\Omega_s) \colon \quad K\mu \approx z^d(\varepsilon) 
\quad\text{where } 
K\mu = \left(\int_\Omega k(x_1,y) ~\mathrm{d}\mu(y); \dots; \int_\Omega k(x_{N_o},y) ~\mathrm{d}\mu(y) \right).
\end{align*}
Here, $\M(\Omega_s)$ is the space of Radon measures defined on the location set $\Omega_s$.
At first glance, this might seem counter-intuitive: The space~$\mathcal{M}(\Omega_s)$ is much larger than the set of ``sparse'' signals of the form~\eqref{def:sparse}. Thus, this lifting should only contribute to the ill-posedness of the problem. However, it also bypasses the nonlinear dependency of~$k(x_j,\cdot)$ onto the location of the sources and enables the use of powerful tools from variational regularization theory for the reconstruction of~$\mu^\dagger$. 
In this work, stable recovery of \(\mu\) is facilitated by a variational Tikhonov estimator in the space of Radon measures~\cite{hofmann_kaltenbacher_poschl_scherzer_2007,bredies_pikkarainen_2012}, which amounts to solving a nonsmooth minimization problem over this space.

However, measurements stemming from experiments are always affected by errors, either due to external influences, imperfectness of the measurement devices or human failure. These have to be taken into account in order to guarantee a stable recovery of~$\mu^\dagger$.
In particular, it is evident that the choice of the measurement locations and the quality of the employed sensors is a key factor for the successful and robust reconstruction of the signal. This directly leads to the problem of sensor design, which is to identify a measurement configuration leading to recovery guarantees with minimal error for the given effort, in a suitable way. Since the sensor design must usually be chosen before the exact source is known and the practical measurement has been performed (thus yielding a realization of the noise), this usually calls for a stochastic framework for the noise.
Although much is know about the error caused by deterministic noise~\cite{granda_2013,duval_peyre_2014,sparsespike,poon_2018,poon_2023}, we are not aware of any works pertaining to the case of stochastic noise in this context.
Moreover, existing deterministic bounds on the error of the recovery \(\mu(\eps)\) to the ground truth \(\mu^\dagger\) are not explicit in terms of the measurement locations \(x_j\) the statistical properties of the error \(\varepsilon_j\) and ground truth \(\mu^\dagger\), and thus can not directly be used to quantify the influence of the measurement locations on the error.
The explicit dependency on the measurement setup is needed to guide the choice of an optimal design that minimizes the expected recovery error for a given cost (often measured in terms of number and quality of sensors), while robustness with respect to the ground truth is desirable if only an approximate guess of the exact source is available (which is the realistic case, in practice).

In addition, to quantify the error, often estimates are given separately in terms of positions and coefficients, which can then be translated into an an upper bound of the error of the measure, which may be an overestimate by a large factor.
To provide a useful bound for sensor placement, we start from the error in the recently developed Hellinger-Kantorovich metric~\cite{LieroMielkeSavare:2018}, which we then link to the parameters and the quantitative bound that is asymptotically sharp.

\subsection{Sparse inverse problems with deterministic noise} 
Despite the popularity of sparse inverse problems, most of the existing work, to the best of our knowledge, focuses on deterministic noise~$\epsilon$. 
Central objects in this context, are the (noiseless) \textit{minimum norm problem}  
\begin{align} \label{def:minprob}
\min_{\mu \in \mathcal{M}(\Omega_s) } \norm{\miu}_{\M(\Omega_s)}  \quad\text{subject to } K\mu=K\mu^\dagger \tag{$\mathcal{P}_0$}
\end{align}
as well as the question whether~$\mu^\dagger$ is \textit{identifiable}, i.e., its unique solution. A sufficient condition for the latter, is, e.g., the injectivity of the restricted operator~$K_{\vert \supp \mu^\dagger}$ as well as the existence of a so-called \textit{dual certificate}~$\eta^{\dagger} \in \mathcal{C}^2(\Omega_s)$, \cite{duval_peyre_2014}, i.e.\ a subgradient~$\eta^\dagger \in \partial \|\mu^\dagger\|_{\mathcal{M}(\Omega_s)}$, which is in some sense minimal, satisfying a~\textit{strengthened source condition} 
\begin{align*}
|\eta^\dagger(y)|\leq 1 \quad \text{for all}~y \in \Omega_s,~\eta^\dagger(y^\dagger_n)= \operatorname{sign}(q^\dagger_n), \quad |\eta^\dagger(y)|< 1 \quad \text{for all}~y \in \Omega_s \setminus \{y^\dagger\}^{N_s}_{n=1}. 
\end{align*}
For example, in particular settings, the groundbreaking paper~\cite{candes_2014} shows that~$\mu^\dagger$ is identifiable if the source locations~$y^\dagger_n$ are sufficiently well separated.
In this context, several manuscripts, see e.g.~\cite{duval_peyre_2014,sparsespike,poon_2023,granda_2013} for a non-exhaustive list, study the approximation of an identifiable~$\mu^\dagger$ by solutions to the Tikhonov-regularized problem
\begin{align} \label{eq:estproblemmeasure0_bis2}
\bar{\mu}(\epsilon) \in \mathfrak{M}(\epsilon) \coloneqq\argmin_{\mu\in\mathcal{M}(\Omega_s)} \left \lbrack \dfrac{1}{2} \norm{\Ke \miu - \zd(\epsilon)}^2_{\Sigma_0^{-1}} + \beta\norm{\miu}_{\M(\Omega_s)} \right\rbrack, \tag{$\mathcal{P}_{\beta, \epsilon}$}
\end{align}
where~$\Sigma_0$ is a positive definite diagonal matrix and the regularization parameter~$\beta=\beta(\|\epsilon\|)>0$ is adapted to the strength of the noise. This represents a challenging \textit{nonsmooth} minimization problem over the infinite-dimensional and non-reflexive space of Radon measures. Moreover, due to its lack of strict convexity, its solutions are typically not unique. Under mild conditions on the choice of $\beta$, arbitrary solutions~$\bar{\mu}(\epsilon)$ approximate~$\mu^\dagger$ in the weak*-sense
as~$\epsilon$ goes to zero.
Moreover, it was shown in~\cite{duval_peyre_2014} that if the minimal dual certificate $\eta^\dagger$ associated to problem~\eqref{def:minprob} satisfies the strengthened source condition and its curvature does not degenerate around~$y^\dagger_n$, $\bar{\mu}(\epsilon)$ is unique and of the form
\begin{align*}
\bar{\mu}(\epsilon)= \sum^{N^\dagger_s}_{n=1} \bar{q}_n (\epsilon) \delta_{\bar{y}_n (\epsilon)}  \quad \text{with} \quad |\bar{q}_n (\epsilon)-q^\dagger_n|+ \norm{\bar{y}_n (\epsilon)-y^\dagger_n}= \mathcal{O}(\|\epsilon\|) 
\end{align*}
provided that~$\|\epsilon\|$ and~$\beta$ are small enough.

\subsection{Sparse inverse problems with random noise}
From a practical perspective, assuming knowledge on the norm of the error is very restrictive or even unrealistic and a statistical model for the measurement error is more appropriate. While the literature on deterministic sparse inversion is very rich, there are only few works dealing with randomness in the problem. We point out, e.g.,~\cite{candesnoisy} in which the authors consider additive  i.i.d.\ noise stemming from a low-pass filtering of the signal. A reconstruction~$\bar{\mu}(\varepsilon)$ is obtained by solving a constrained version of~\eqref{eq:estproblemmeasure0_bis2} and the authors show that, with high probability, there holds~$Q_{\text{hi}} (\bar{\mu}(\varepsilon)) \approx Q_{\text{hi}} (\mu^\dagger) $ where~$Q_{\text{hi}}$ is a convolution with a high-resolution kernel. Moreover, in~\cite{poon_2023} the authors consider deterministic noise but allow for randomness in the forward operator~$K$. Their main result provides an estimate on an optimal transport energy between two positive measures derived from source and reconstruction. These again hold with high probability. Finally, we also mention~\cite{Engel} in which the authors propose a first step towards \textit{Bayesian inversion} for sparse problems, i.e.\ both measurement noise as well as the unknown~$\mu^\dagger$ are considered to be random variables. A suitable prior is constructed and well-posedness of the associated Bayesian inverse problem is shown. 

In this paper, similar to~\cite{candesnoisy}, we adopt a frequentist viewpoint on sparse inverse problems and assume that the measurement errors follow a known probability distribution. In contrast, the unknown signal~$\mu^\dagger$ is treated as a deterministic object. More in detail, we assume unbiased independent Gaussian noise with diagonal covariance matrix \(\Sigma = \diag(\sigma_j)\), corresponding to independent measurements with variable quality sensors at different locations. 
We consider the Tikhonov-type estimator~\eqref{eq:estproblemmeasure0_bis2} with 
\[
\Sigma^{-1}_0 = \Sigma^{-1} / p, \quad\text{where } p = \tr(\Sigma^{-1})
\]
and investigate its error to the ground truth, where we have to account for the randomness of the noise. In statistical terms, \(\Sigma^{-1}\) is the precision matrix of the sensor array, and \(p\) can be interpreted as an overall precision of the combined measurement, roughly representing an analogue to \(1/\norm{\varepsilon}\) in the stochastic setting.
First and foremost, the uncertainty of the noise propagates to the estimator and thus~$\bar{\mu}$ has to be interpreted as a random variable.
Second, unlike the deterministic setting of~\cite{duval_peyre_2014}, the asymptotic analysis cannot exclusively rely on smallness assumptions on the Euclidean norm of the noise: some realizations of~$\varepsilon$ might be very large, albeit with small probability.
Consequently,  reconstructions can exhibit undesirable features such as clustering phenomena around~$y^\dagger_n$ or spurious sources far away from the true support. In particular, the reconstructed signal may comprise more or less than~$N^\dagger_s$ sources.
Thus, we require a suitable distance between signed measures that is compatible with weak* convergence on bounded subsets of~$\M(\Omega_s)$.
We find a suitable candidate in generalizations of optimal transport energies~\cite{LieroMielkeSavare:2018}; cf.\ also~\cite{chizat,poon_2023}.

Despite its various difficulties, stochastic noise also provides new opportunities. For example, unlike the deterministic case, we are given a whole distribution of the measurement data and not only one particular realization. Clearly, the uncertainty in the estimate critically depends on the appropriate choice of measurement locations \(\bld{x} = (x_j)_{j=1,\ldots,N_o}\), the overall precision \(p\), and relative precision of each sensor \(\Sigma_0^{-1}\).
Formalizing this connection enables the mathematical program of~\textit{optimal sensor placement} or optimal design, i.e.\ an optimization of the measurement setup to mitigate the influence of the noise before any data is collected in a real experiment.
This requires a cheap-to-evaluate~\textit{design criterion} which allows to compare the quality of different sensor setups. For linear inverse problems in Hilbert spaces, a popular performance indicator is the mean-squared error of the associated least-squares estimator, which admits a closed form representation through its decomposition into variance and bias; see, e.g.,~\cite{haber_2012}. For nonlinear problems, \textit{locally optimal} sensor placement approaches rely on a linearization of the forward model around a best guess for the unknown parameters; see, e.g.,~\cite{dariusz}.   
To the best of our knowledge, optimal sensor placement for nonsmooth estimation problems and for infinite dimensional parameter spaces beyond the Hilbert space setting is uncharted territory.

\subsection{Contribution} 
Taking the mentioned difficulties in the stochastic setting into consideration, we are led to the analysis of the~\textit{worst-case mean-squared-error} of the estimator
\begin{align} \label{eq:MSE}
\operatorname{MSE}[\bar{\mu}]
\coloneqq \mathbb{E} \left \lbrack \sup_{ \bar{\mu} \in \mathfrak{M}} d_{\HK}(\bar{\mu},\mu^\dagger )^2\right \rbrack
= \int_{\R^{N_o}} \sup_{ \bar{\mu} \in \mathfrak{M}(\epsilon)} d_{\HK}(\bar{\mu},\mu^\dagger )^2\de\gamma_p(\epsilon),
\end{align}
where~$d_{\text{HK}}$ denotes an extension of the Hellinger-Kantorovich distance introduced in~\cite{LieroMielkeSavare:2018} to signed measures (see Section~\ref{sec:assessing_reconstruction_errors}) and \(\gamma_p\) is the noise distribution \(\mathcal{N}(0,\Sigma)\).  We point out that, in comparison to linear inverse problems in Hilbert space,~$\operatorname{MSE}[\bar{\mu}]$ does not admit a closed form expression and its computation requires both, a sampling of the expected value, as well as an efficient way to calculate the Hellinger-Kantorovich distance. This prevents its direct use in the context of optimal sensor placement for sparse inverse problems.

To enable efficient sensor design, we first need to select an appropriate regularization parameter, depending on the noise level.
Here, we focus on the \emph{a~priori} choice rule of~$\beta(p) = \beta_0 / \sqrt{p}$ for some tunable~$\beta_0 > 0$, that only takes into account the overall precision of the sensor. 
For this choice, we provide the following upper bound:
\begin{align} \label{eq:mainintro}
    \MSE[\bar{\mu}]
    \leq  \frac{8}{p} \, \psi_{\beta_0}(\bld{x},\Sigma_0) 
    + \Bar{c}  \exp \left(-\Bar{\lambda}\beta_0^2\right),
\end{align}
where the constant \(\psi_{\beta_0}(\bld{x},\Sigma_0)\) (further detailed below) explicitly depends on the locations and relative precisions while the constants \(\Bar{c}\) and \(\Bar{\lambda}\) depend on the problem setup (the kernel and domain, some basic bounds on the ground truth), a non-degeneracy parameter of the dual certificate \(\eta^\dagger\) (further detailed below), and quantities that can be bounded by \(\psi_{\beta_0}(\bld{x},\Sigma_0)\), but do not depend on \(p \) or \(\beta_0\) for \(p \geq \bar{p}>0\) and \(\beta_0 \geq \bar{\beta}_0>0\); see Theorem~\ref{thm:MSE_bound}.
Thus, under these basic assumptions and by choosing~$\beta_0$ large enough, the second term in~\eqref{eq:mainintro} becomes negligible and the first term dominates and closely predicts the mean-squared error.
This behavior is confirmed by numerical examples; see Section~\ref{sec:numerics}.

To further illustrate the meaning of the constant \(\psi_{\beta_0}(\bld{x},\Sigma_0)\), let us denote by $\bld{q} = (q_1;\ldots; q_{N_s})$ and $\bld{y} = (y_1;\ldots; y_{N_s})$ the vectors of coefficients and positions of sources, respectively.
Additionally, we collect all the parameters of a given finite source \(\mu\) in the vector~$\bld{m} = (\bld{q}; \bld{y})$, and introduce the \emph{parameter-to-observation} map \(G(\bld{m}) = K\mu\),
as well as its Jacobian \(G'(\bld{m}^\dagger)\) evaluated at the parameters of the ground truth.
Associated to this, we denote the Fisher information matrix~$\mathcal{I}_0$ by
\begin{align} \label{eq:deffishersign}
    \mathcal{I}_0 := G'(\bld{m}^\dagger)^\top \Sigma_0^{-1} G'(\bld{m}^\dagger) .
\end{align}
Then the constant in the estimate above is computed as
\begin{align*}
    \psi_{\beta_0}(\bld{x},\Sigma_0)
    = \tr(W_{\dagger}\mathcal{I}_0^{-1}) + \beta_0^2 \norm{\mathcal{I}_0^{-1} \vectorsgn}^2_{W_{\dagger}}, 
\end{align*}
with the sign vector~$\bld{\rho} = \sign \bld{q}^\dagger$ and a weighted Euclidean norm~$\|\cdot\|_{W_\dagger}$ which is induced by a positive definite matrix~$W_\dagger$ connected to the ground truth~$\bld{m}^\dagger$. 
This clarifies how the multiplicative constant in the estimate
explicitly depends on the measurement setup and we note that it closely resembles the ``classical'' A-optimal design criterion; cf.~\cite{haber_2012}. Together with the estimate~\eqref{eq:mainintro}, and the smallness of the second term, this suggests that~$\psi_{\beta_0}(\bld{x},\Sigma_0)$ is a suitable criterion to quantify the quality of a given design in terms of the MSE~\eqref{eq:MSE}.


Concerning the smallness of the second term, we note that the constant \(\Bar{\lambda}\) also depends on a non-degeneracy constant \(\theta>0\), which is a further tightening of the assumption on the dual certificate.
This non-degenerate source condition on \(\mu^\dagger\) requires the associated minimal norm dual certificate~$\eta^\dagger$ to fulfill
\begin{equation}\label{eq:nondegeneracyintro}
\abs{\eta^\dagger(y)} \leq 1 - \theta \min\left\{\theta,\, \min_{n=1,\ldots,N_s} \norm{\sqrt{|q_n^{\dagger}|}(y-y_n^{\dagger})}_{2}^2 \right\}
\quad\text{for all } y\in \Omega_s
\end{equation}
for some~$\theta >0$. This condition has been employed in many previous works, and is known to uniformly hold for several settings under general assumptions on the measurement and a separation of the condition of the sources; see, e.g.~\cite{poon_2018} and the references therein.

The proof of the main result relies on a splitting of the set of measurement errors~$\R^{N_o}$ into a set of ``nice'' events~$\mathcal{A}_{\text{nice}}$ as well as an  estimate of the probability of its complement~$\R^{N_o} \setminus \mathcal{A}_{\text{nice}}$, related to the second term in~\eqref{eq:mainintro}.
On~$\mathcal{A}_{\text{nice}}$, there is a unique optimal parameter~$\bar{\bld{m}}=(\bar{\bld{q}},\bar{\bld{y}})$ with the correct number of sources that parametrizes \(\bar{\mu}\).
Then, the distance between the reconstruction and the ground truth in the Hellinger-Kantorovich distance can be estimated by a weighted Euclidean distance of the parameters. Those can be further estimated with a linearization of \(G\), which leads to~\eqref{eq:mainintro} after explicitly computing the expectation.
This estimate is specific to the choice of~$d_{\text{HK}}$ and relies on its interpretation as an unbalanced Wasserstein-2 distance. While similar estimates can be derived for other popular metrics such as the Kantorovich-Rubinstein distance (related to the Wasserstein-1 distance; see Appendix~\ref{app:distance_metric}) this would introduce additional constants in the first term of~\eqref{eq:mainintro} stemming from an inverse inequality of discrete \(\ell_1\) and weighted \(\ell_2\) norms.
Thus, the first term in the modified estimate would overestimate the true error by a potentially substantial factor.
In contrast, the first term in~\eqref{eq:mainintro} is sharp in the sense that the convenient factor of~$8$ can, mutatis mutandis, be replaced by any~$c > 1$, at the cost of increasing the constant in the second term.

\subsection{Further related work}
\subsubsection*{Sparse minimization problems beyond inverse problems}
Minimization problems over spaces of measures represent a sensible extension of~$\ell_1$-regularization towards decision variables on continuous domains. Consequently, problems of the form~\eqref{eq:estproblemmeasure0_bis2} naturally appear in a variety of different applications, detached from inverse problems. We point out, e.g., optimal actuator placement, optimal sensor placement~\cite{neitzel_2019}, as well as the training of shallow neural networks~\cite{bach}. Non-degeneracy conditions similar to~\eqref{eq:nondegeneracyintro} play a crucial role in this context and form the basis for an in-depth (numerical) analysis of the problem, e.g., concerning the derivation of fast converging solution methods,~\cite{chizat,flinth,pieper_walter_2021}, or finite element error estimates~\cite{vexler}.   
\subsubsection*{Inverse problems with random noise}
Frequentist approaches to inverse problems have been studied previously in, e.g.,~\cite{gerth_2017,werner_2012}. These works focus on the ``lifting'' of deterministic regularization methods as well as of their consistency properties and convergence rates to the random noise setting. This only relies on minimal assumptions on the inverse problem, e.g., classical source conditions, and thus covers a wide class of settings. Similar to the present work, an important role is played by a splitting of the possible events into a set on which the deterministic theory holds and its small complement. However, we want to stress that the proof of the main estimate in~\eqref{eq:mainintro} is problem-taylored and relies on exploiting specific structural properties of inverse problems in spaces of measures. Moreover, our main goal is~\textit{not} the consistency analysis of an estimator but the derivation of a useful and mathematically sound design criterion for sparse inverse problems.

\subsection*{Organization of the paper.}
The paper is organized as follows: In Section \ref{sec:sparseinverse}, we recall some properties of the minimum norm problem~\eqref{eq:estproblemmeasure0_bis2} and the Tikhonov regularized problem \eqref{eq:estproblemmeasure0_bis2} as well as its solutions. In Section \ref{sec:assessing_reconstruction_errors}, we define the Hellinger-Kantorovich distance and investigate its properties. Section \ref{sec:sensorplacement} is devoted to study the linearized estimate $\delta \widehat{\bld{m}}$. Using these results, we then investigate sparse inverse problems with random noise in Section~\ref{sec:stochasticnoise} and provide a sharp upper bound for $\MSE[\bar{\mu}]$ in Section \ref{sec:quantativeerror}. Finally, in Section \ref{sec:numerics} we present some numerical examples to verify our theory.

\section{Notation and preliminaries} \label{sec:notandprelim}

Before going into the main part of the paper, we introduce the basic notation used throughout the paper and gather preliminary assumptions concerning the considered integral kernels as well as pertinent facts on Radon measures.

\subsection{Notation}
Throughout the paper,~$c_i, C_i$, $i = 1,2,\ldots$ denote generic constants that may vary from line to line. By $C = C(a,b,\ldots)$, we indicate that $C$ depends on $a,b,\ldots$. We denote by \(\Omega_s \subset \R^{d}\) and \(\Omega_o \subset \R^{d_o}\) the compact location and observation set, where \(d_o, d \geq 1\) and~$\Omega_s$ has a nonempty interior. A vector in $X^m$ for a set $X$ and $m > 1$, will be written in bold face, for instance $\bld{y} = (y_1;\ldots;y_{N_s}) \in \Omega_s^{N_s}$, $\bld{q} = (q_1;\ldots;q_{N_s}) \in \R^{N_s}$ and $\bld{x} = (x_1;\ldots; x_{N_o}) \in \Omega_o^{N_o}$ are vectors of coefficients, positions of sources and positions of observations, respectively, where the formal definitions are introduced in the sequel. We write $(\bld{a}_1, \ldots, \bld{a}_n)$ and $(\bld{a}_1;\ldots; \bld{a}_n)$ to stack vectors $\bld{a}_1, \ldots, \bld{a}_n$ horizontally and vertically, respectively.  We write $\norm{\cdot}_p$ for the usual $\ell^p$-norm on $\R^m$. For a vector $x \in \R^m$ and a positively defined matrix $W \in \R^{m \times m}$, we define the weighted $W$-norm of $x$ as $\norm{x}_{W}:= \norm{W^{1/2}x}_2$. The closed ball in this weighted norm is denoted by $B_W(x,r) := \{\, x' \in \R^m\,:\, \norm{x' - x}_W \le r \,\}$. For a linear map $A: X \to Y$, the operator norm of $A$ is given by $\norm{A}_{X \to Y} = \sup_{ \norm{x}_X \le 1} \norm{Ax}_Y$.  Similarly, any bilinear map $A: X_1 \times X_2 \to Y$ has a natural operator norm $\norm{A}_{X_1 \times X_2 \to Y}:= \sup_{\norm{x_1}_{X_1} \le 1, \norm{x_2}_{X_2} \le 1} \norm{A(x_1,x_2)}_{Y}.$ 

Furthermore, let $k : \Omega_o \times \Omega_s \to \R$ be a real-valued kernel.  We introduce the following notations which turn $k$ into vector-valued kernels:  $k[\bld{x}](y) = k[\bld{x},y]$ is a column vector with 
\begin{equation} \label{eq:evalkx}
k[\bld{x},y] := \left(k(x_1,y);\ldots;  k(x_{N_o}, y)\right), \quad \bld{x} = (x_1;\ldots; x_{N_o}) \in \Omega_o^{N_o}, \quad y \in \Omega_s,
\end{equation}
while $k[x,\bld{y}]$ is a row vector with
\begin{equation} \label{eq:evalky}
    k[x,\bld{y}] := (k(x,y_1), \ldots, k(x,y_{N_s})), \quad x \in \Omega_o, \quad \bld{y} = (y_1;\ldots; y_{N_s}) \in \Omega_s^{N_s}.
\end{equation}
Similarly, we also have the matrix $k[\bld{x}, \bld{y}]$ defined as
\begin{equation} \label{eq:evalkxy}
k[\bld{x}, \bld{y}] := (k(x_1, \bld{y}); \ldots; k(x_{N_o}, \bld{y})). 
\end{equation}
When $k = k(x,\cdot)$ is a smooth function in variable $y$,  we consider the $r^{\text{th}}$-derivative of $k$ 
the tensor of partial derivatives is \(y\) by \(\nabla^r_{y\cdots y} k(x,y)\).
In particular,  $\nabla_y k(x,y)$ and $\nabla^2_{yy} k(x,y)$ are the gradient and Hessian of $k$ (with respect to variable $y$,) respectively. We note that \(\nabla_y k \colon \Omega_o \times\Omega_s \to \R^{N_s}\) is a vector valued kernel and thus we define \(\nabla_y^\top k[x,\bld{y}]\) as a matrix defined by
\begin{align} \label{eq:evalnabky}
\nabla_y^\top k[x,\bld{y}] 
& = (\nabla_{y}k(x,y_1)^\top, \nabla_{y}k(x,y_2)^\top,\ldots, \nabla_{y}k(x,y_{N_s})^\top).
\end{align}
Similarly, $\nabla_y^\top k[\bld{x},\bld{y}]$ is a block matrix defined by
\begin{align} \label{eq:nabkxy}
\nabla_y^\top k[\bld{x},\bld{y}] 
& = (\nabla_y^\top k[x_1,\bld{y}], \ldots, \nabla_y^\top k[x_{N_o},\bld{y}]).
\end{align}
Throughout the paper, by a slight abuse of notation, we denote by \(\varepsilon\) a variable deterministic noise, a random variable, or its realization, which will be clear from the context. 
By \(\gamma_p\) we denote the density of a multivariate Gaussian random variable with expectation zero and covariance~\(\Sigma\).
Further notation, specific to the present manuscript, will be introduced at first appearance. For quicker reference, a notation table can be found in Appendix~\ref{appendixnotation}.

\subsection{Preliminaries}
We also recall some basic facts and assumptions for inverse source location.
\subsubsection*{Integral kernels} \label{subsec:kernels}

Throughout the paper, we assume that the kernel is sufficiently regular:
\begin{enumerate}[resume,label=\bf{(A\arabic{enumi}}),ref=A\arabic{enumi}]
\item \label{ass:kernel}
The kernel $k \in \mathcal{C}(\Omega_o \times \Omega_s)$ is three-times differentiable in the variable $y$. For abbreviation, we further set
\end{enumerate}
\begin{align*}
C_{\ke} &:= \sup_{x\in\Omega_o,y\in\Omega_s} \abs{\ke(x,y)}, &
C'_{\ke} &:= \sup_{x\in\Omega_o,y\in\Omega_s} \norm{\nabla_y \ke(x,y)}_2, \\
C''_{\ke} &:= \sup_{x\in\Omega_o,y\in\Omega_s} \norm{\nabla^2_{yy} \ke(x,y)}_{2 \to 2}, &
C'''_{\ke} &:= \sup_{x\in\Omega_o,y\in\Omega_s} \norm{\nabla^3_{yyy} \ke(x,y)}_{2 \times 2 \to 2}.
\end{align*}
By means of the kernel $k$, we introduce the weak* continuous \emph{source-to-measurements} operator~$K \colon \mathcal{M}(\Omega_s) \to \R^{N_o}$ with
\begin{align} \label{eq:sourcetomeas}
K\mu = \left(\int_{\Omega_s}k(x_1,y)\de \mu(y);\ldots;\int_{\Omega_s}k(x_{N_o},y)\de \mu(y)\right).
\end{align}
Moreover, consider the operator $K^* \colon \R^{N_o} \to \mathcal{C}^2(\Omega_s)$ given by
\begin{align} \label{eq:preadjoint}
    \lbrack K^*z \rbrack (y)=\sum^{N_o}_{j=1} z_j k(x_j,y) \quad \text{for all} \quad z \in \R^{N_o}.  
\end{align}
Then~$K^*$ is linear and continuous and there holds
\begin{align*}
    \int_\Omega \lbrack K^*z \rbrack (y)~\mathrm{d}\mu(y)= z^\top \lbrack K\mu \rbrack  \quad \text{for all}\quad \mu \in \mathcal{M}(\Omega_s),~z \in \R^{N_o}.
\end{align*}
\subsubsection*{Space of Radon measures.} \label{subsec:Radon} We recall some properties of Radon measures. Let $\Omega \subset \R^d$,  $d \ge 1$ be a compact set.  We define the space of Radon measures $\M(\Omega)$  as the topological dual of the space $\mathcal{C}(\Omega)$ of continuous functions on $\Omega$ endowed with the supremum norm. It is then a Banach space equipped with the dual norm
\[
\norm{\mu}_{\M(\Omega)} := \sup \left\{ \int_{\Omega} f \de \mu: f \in \mathcal{C}(\Omega), \norm{f}_{\mathcal{C}(\Omega)} \le 1 \right\}.
\]
Weak* convergence of a sequence in~$\mathcal{M}(\Omega)$ will be denoted by ``$\rightharpoonup^*$''. More specifically, we have
\[
\mu_n \rightharpoonup^* \mu \quad\text{ if and only if} \quad \int_{\Omega} f \de \mu_n \to \int_{\Omega} f \de \mu \quad \text{for all} \quad f \in \mathcal{C}(\Omega).
\]
Next, by the definition of the total variation norm, its subdifferential is defined by
\[
\partial \norm{\miu}_{\M(\Omega_s)} := \left\{ \eta \in \mathcal{C}(\Omega_s): |\eta(y)| \le 1,\forall y \in \Omega_s \text{ and } \int_{\Omega_s} \eta \de \miu = \norm{\miu}_{\M(\Omega_s)} \right\},
\]
see for instance \cite{duval_peyre_2014}. In particular, for a discrete measure $\miu = \sum_{n=1}^{N} q_n\delta_{y_n}$ one has
\begin{align*}
    \partial \norm{\miu}_{\M(\Omega_s)} = \left\{ \eta \in \mathcal{C}(\Omega_s): |\eta(y)| \le 1,\forall y \in \Omega_s \text{ and } \eta(y_n) = \sign(q_n),\forall n = 1,\ldots, N \right\}.
\end{align*}
Finally, by $\M^+(\Omega)$ we refer to the set of positive Radon measures on $\Omega$.




\section{Sparse inverse problems with deterministic noise} 
\label{sec:sparseinverse}
Our interest lies in the stable recovery of a sparse ground truth measure 
\begin{align*}
\mu^\dagger= \sum^{N^\dagger}_{n=1} q^\dagger_n \delta_{y^\dagger_n} \quad \text{for some} \quad q^\dagger_n \in \R,~
\end{align*}
by solving the Tikhonov regularization~\eqref{eq:estproblemmeasure0_bis2} associated to the inverse problem~$z^d=K\mu$ given noisy data~$z^d$.
In this preliminary section, we give some meaningful examples of this abstract setting and briefly recap the key concepts and results in the case of additive deterministic noise
\begin{equation} \label{eq:defmeasurements}
    z^d(\epsilon)=K\mu^\dagger+\epsilon \quad \text{for some}\quad \epsilon \in \R^{N_o}. 
\end{equation}
In particular, we clarify the connection between~\eqref{eq:estproblemmeasure0_bis2} and~\eqref{def:minprob} and recall a first qualitative statement on the asymptotic behavior of solutions to~\eqref{eq:estproblemmeasure0_bis2} for a suitable a priori regularization parameter choice~$\beta=\beta(\epsilon)$. 
\subsection{Examples} 
Sparse inverse problems appear in a variety of interesting applications. In the following, we give some examples which fit into our setting.


\begin{example} \label{ex:2}
Consider the advection-diffusion equation
\begin{equation}\label{ex:adv}
\partial_t u - \nabla(\bld{D}\cdot \nabla u) + \nabla \cdot (\kappa u) = 0 \text{ in }(0,T) \times \R^d,
\end{equation}
together with the initial value $u(0,\cdot) = \mu$. The boundary condition is given by $u \to 0$ as $x \to \infty$. This equation describes the rate of change of the concentration of the contaminant $u(t,x)$. For simplicity, we consider a two-dimensional medium, and both $\kappa= (\kappa_1,\kappa_2)$ and $\bld{D} = \diag(D_1,D_2)$ are independent of $x$. 
Here the solution to \eqref{ex:adv} is given by
\[
u(t,x) = \int_{\R^2} G(x - y, t) \de \mu(y) 
\]
where $G(x,t)$ is the Green's function of the advection-diffusion equation, which is given by
\[
G(x,t) = \frac{1}{4\pi \sqrt{D_{1}D_{2}t}}
\exp(-\norm{x - \kappa t}_{\bld{D}^{-1}}^2/(4t)).
\]
Here, if one seeks to identify the initial value $\mu$ from finite number of measurements at time $T_o > 0$ in the observation set $\Omega_o \subset \R^2$, the kernel is given by \(\ke(x,y) = G(x-y, T_o)\).
\end{example}

\begin{example}Consider the advection-diffusion equation on a bounded smooth domain \(\Omega\), together with the Dirichlet boundary conditions \(u\rvert_{(0,T) \times \partial \Omega} = 0\), then there exists a kernel $G(x,y,t)$ such that
\[
u(t,x) = \int_{\Omega} G(x, y, t) \de \mu(y),
\]
see, e.g., \cite{Fukushima_2010}.  In this case, for observations at time \(T_o\) we choose \(\ke = G(\cdot, \cdot, T_o)\).  For \(\Omega_o \subset \Omega\) (i.e., no observation near the boundary), the regularity requirements on \(\partial\Omega\) are not necessary since one can employ interior regularity arguments; see, e.g., \cite{grisvard_2011}.
\end{example}

\subsection{Tihkonov regularization of sparse inverse problems }
In this section, we briefly summarize some preliminary results concerning the regularized problem~\eqref{eq:estproblemmeasure0_bis2} as well as its solution set. We start by discussing its well-posedness.

\begin{proposition}
\label{prop:bound_on_barq} Problem~\eqref{eq:estproblemmeasure0_bis2} admits a solution $\bar{\miu}$. Furthermore, any solution $\bar{\mu}$ to \eqref{eq:estproblemmeasure0_bis2} satisfies 
\(\norm{\bar{\miu}}_{\mathcal{M}(\Omega_s)}
    \leq \norm{\epsilon}_{\Prec}^2/(2\beta) + \norm{\miu^\dagger}_{\mathcal{M}(\Omega_s)}\)
    and the solution set
\begin{align*}
  \mathfrak{M}(\epsilon)  =\argmin \eqref{eq:estproblemmeasure0_bis2}
\end{align*}
is weak* compact.
\end{proposition}
\begin{proof}
Existence of a minimizer of~\eqref{eq:estproblemmeasure0_bis2} is guaranteed by~\cite[Proposition 3.1]{bredies_pikkarainen_2012} noticing that the forward operator $K: \M(\Omega_s) \to \R^{N_o}$ of \eqref{eq:estproblemmeasure0_bis2} is weak*-to-strong continuous. For the upper bound we use the optimality of~$\bar{\miu}$ compared to \(\miu^\dagger\) as well as the definition of~$z^d(\epsilon)$ to get
\[
\beta\norm{\bar{\miu}}_{\mathcal{M}(\Omega_s)}\leq \frac{1}{2}\norm{\Ke \bar{\miu} - \zd}_{\Prec}^2
+ \beta\norm{\bar{\miu}}_{\mathcal{M}(\Omega_s)}
\leq \frac{1}{2}\norm{\epsilon}_{\Prec}^2
+ \beta\norm{\miu^\dagger}_{\mathcal{M}(\Omega_s)}.
\]
Moreover,~$\mathfrak{M}(\epsilon)$ is weak* closed since the objective functional in~\eqref{eq:estproblemmeasure0_bis2} is weak* lower semicontinuous. Combining both observations, we conclude the weak* compactness of~$\mathfrak{M}(\epsilon)$. 
\end{proof}
In particular, note that~$\mathfrak{M}(\epsilon)$ is, in general, not a singleton due to the lack of strict convexity in~\eqref{eq:estproblemmeasure0_bis2}. Moreover, we recall that the inverse problem was introduced as a lifting of the nonconvex and combinatorial integral equation~\eqref{eq:inverse}. From the same perspective,~\eqref{eq:estproblemmeasure0_bis2} can be interpreted as a convex relaxation of the parametrized problem
\begin{align} \label{def:estproblempoints}
\inf_{\substack{\bld{y}\in \Omega^N_s,~\bld{q} \in \mathbb{R}^N, \\ N \in \mathbb{N}}} \left \lbrack \dfrac{1}{2} \norm{\ke[\bld{x},\bld{y}]\bld{q} - \zd}_{\Prec}^2 + \beta \norm{\bld{q}}_{1} \right\rbrack,
\end{align}
In the following proposition, we show that this relaxation is exact, i.e.\ there exists at least one solution to~\eqref{def:estproblempoints} and its minimizers parametrize sparse solutions to~\eqref{eq:estproblemmeasure0_bis2}.
\begin{proposition}
There holds~$\min\eqref{eq:estproblemmeasure0_bis2} = \inf \eqref{def:estproblempoints}$.
For a triple $(\bar{N},\bar{\bld{y}},\bar{{\bld{q}}})$ with~$\bar{y}_i \neq \bar{y}_j $,~$i \neq j$, the following statements are equivalent:
\begin{itemize}
    \item The triple $(\bar{N},\bar{\bld{y}},\bar{{\bld{q}}})$ is a solution of~\eqref{def:estproblempoints}.
\item The parametrized measure~$\bar{\mu}=\sum^{\bar{N}}_{n=1} \bar{q}_n \delta_{\bar{y}_n}$ is a solution of~\eqref{eq:estproblemmeasure0_bis2}.
\end{itemize}
Moreover,~\eqref{eq:estproblemmeasure0_bis2} admits at least one solution of this form with~$\bar{N}\leq N_o$. 
\end{proposition}

\begin{proof}
Given $({N},{\bld{y}},{{\bld{q}}})$ with~${y}_i \neq {y}_j $,~$i \neq j$, note that the sparse measure
\begin{align*}
    \mu({\bld{y}},{{\bld{q}}})=\sum^N_{n=1} q_n \delta_{y_n} \quad \text{satisfies} \quad K\mu({\bld{y}},{{\bld{q}}})=\ke[\bld{x},\bld{y}]\bld{q},~ \|\mu({\bld{y}},{{\bld{q}}})\|_{\mathcal{M}(\Omega_s)}=\norm{\bld{q}}_{1}.
\end{align*}
Hence, one readily verifies~$\min\eqref{eq:estproblemmeasure0_bis2} = \inf \eqref{def:estproblempoints}$ as well as the claimed equivalence due to the weak* density of the set of sparse measures in~$\mathcal{M}(\Omega_s)$ and since the objective functional in~\eqref{eq:estproblemmeasure0_bis2} is weakly* lower semicontinuous. The existence of a sparse solution to~\eqref{eq:estproblemmeasure0_bis2} follows similarly to \cite[Theorem 3.7]{pieper_tang_trautmann_walter_2020}.  
\end{proof}
The equivalence between both of these problems will play a significant role in our subsequent analysis. Additional insight on the structure of solutions to~\eqref{eq:estproblemmeasure0_bis2} can be gained  through the study of its first order necessary and sufficient optimality conditions. Since our interest lies in sparse solutions, we restrict the following proposition to this particular case.
\begin{proposition}\label{prop:first_order} A measure $\bar{\miu} = \sum_{n=1}^{\bar N} \bar{q}_n \delta_{\bar{y}_n}$ is a solution of \eqref{eq:estproblemmeasure0_bis2} if and only if 
\[
|\bar{\eta}(y)| \leq 1
\;\text{for all } y \in \Omega_s,
\quad
\bar{\eta}(\bar{y}_n)  = \operatorname{sign}(\bar{q}_n), \quad \forall n = 1, \ldots, \bar{N},
\]
where 
\begin{equation} \label{eq:defdualreg}
\bar{\eta} = -\Ke^*\Prec(\Ke \bar{\miu} - \zd)/\beta
= \Ke^*\Prec \left(\zd -  \ke[\bld{x},\bar{\bld{y}}]\bar{\bld{q}}\right)/\beta.    
\end{equation}
\end{proposition}
Note that~$\bar{\eta}$ is independent of the particular choice of the solution to~\eqref{eq:estproblemmeasure0_bis2}. We will refer to it as the dual certificate associated to~\eqref{eq:estproblemmeasure0_bis2} in the following.
Finally, we give a connection between~\eqref{eq:estproblemmeasure0_bis2} and the minimum norm problem~\eqref{def:minprob} in the vanishing noise limit. The following general convergence property follows directly from \cite{hofmann_kaltenbacher_poschl_scherzer_2007}.
\begin{proposition} \label{prop:convqual}
Assume that~$\beta=\beta(\epsilon)$ is chosen such that
\begin{align*}
 \beta \to 0 \text{ and } \dfrac{\norm{\epsilon}_{\Prec}^2}{\beta} \to 0 \text{ as } \norm{\epsilon}_{\Prec} \to 0.
 \end{align*}
Then solutions to~\eqref{eq:estproblemmeasure0_bis2} subsequentially converge weakly-* towards solutions of~\eqref{def:minprob}. 
\end{proposition}

\subsection{Radon minimum norm problems} \label{subsec:minnorm}
Following Proposition~\ref{prop:convqual}, guaranteed recovery of the ground truth measure 
 requires that~$\mu^\dagger$ is identifiable, i.e.\ the unique solution of~\eqref{def:minprob}. In this section, we briefly summarize some key concepts regarding~\eqref{def:minprob} and state sufficient assumptions for the latter. For this purpose, introduce the associated Fenchel dual problem
\begin{equation}
\label{eq:dual_eta_0}
   \min_{\zeta \in \R^{N_o}} \left \lbrack  -\langle \miu^\dagger, \Ke^* \Prec \zeta \rangle + \mathbb{I}_{\|\Ke^* \Prec \zeta\|_{C(\Omega_s)}\leq 1} \right\rbrack.
\end{equation}
as well as the minimal-norm dual certificate
\begin{equation} \label{eq:defminimalcert}
 {\eta^{\dagger} := K^*\Sigma^{-1}_0 \zeta^\dagger \in \mathcal{C}^2(\Omega_s)}   \quad \text{where} \quad
 \zeta^\dagger = \argmin_{\zeta \in \R^{N_o}} \{ \,\|\zeta\|_2\;:\; \zeta \in \argmin \eqref{eq:dual_eta_0}\,\}.
 \end{equation}
 Note that the existence of $\zeta^{\dagger}$, and therefore the minimum-norm dual certificate $\eta^{\dagger}$, is guaranteed in this setting following~\cite[Proposition A.2]{pieper_tang_trautmann_walter_2020} as well as due to~$K^* \colon \R^{N_o} \to \mathcal{C}^2(\Omega_s)$. Moreover, by standard results from convex analysis, a given~$\mu \in \mathcal{M}(\Omega_s)$ is a solution to~\eqref{def:minprob} if and only if~$\eta^\dagger \in \partial \|\mu\|_{\mathcal{M}(\Omega_s)}$. The following assumptions on~$\mu^\dagger$ and~$\eta^\dagger$ are made throughout the paper:
 \begin{enumerate}[resume,label= \textbf{(A\arabic{enumi})},ref=A\arabic{enumi}]
\item \label{ass:source} \textit{Structure of~$\mu^\dagger$}: We assume that there holds
\begin{equation*}
\miu^\dagger=\sum^{N_s^{\dagger}}_{n=1} q^\dagger_n \delta_{y^\dagger_n} \quad \text{where} \quad q^\dagger_n \neq 0,~y^\dagger_n \in \operatorname{int}(\Omega_s) \quad \text{for all} \quad n = 1, \ldots, N_s^{\dagger}.
\end{equation*}
\item \label{ass:sourcecond}\textit{Source condition}: We assume that the minimum-norm dual certificate $\eta^{\dagger}$ satisfies
\begin{align*}
    |\eta^\dagger(y)| \leq 1 \quad \text{for all} \quad y \in \Omega_s \quad \text{and} \quad \eta^\dagger(y^\dagger_n)= \sign(q^\dagger_n) \quad \text{for all} \quad n=1,\dots,N_s.
\end{align*}
\item \label{ass:sourcecondstrenght} \textit{Strengthened source condition}: We assume that
\begin{align*}
    |\eta^\dagger(y)| < 1 \quad \text{for all} \quad y \in \Omega_s \setminus \{y^\dagger_n\}^{N^\dagger_s}_{n=1}
\end{align*}
and the operator~$K_{|\supp \mu^\dagger} \coloneqq k[\bld{x},\bld{y}^\dagger]$ is injective.
\end{enumerate}
Here, Assumption~\ref{ass:sourcecond} is equivalent to~$\eta^\dagger \in \partial \norm{\mu^\dagger}_{\mathcal{M}(\Omega_s)}$, i.e.,~$\mu^\dagger$ is indeed a solution to~\eqref{def:minprob}, whereas Assumptions~\ref{ass:source} and~\ref{ass:sourcecondstrenght} imply its uniqueness. While Assumption \ref{ass:sourcecondstrenght} seems very strong at first glance, it can be explicitly verified in some settings (see, e.g.,~\cite{candes_2014}) and is often numerically observed in practice. According to~\cite[Proposition~5]{duval_peyre_2014} we have the following:
\begin{proposition}
Let Assumptions~\ref{ass:source}--\ref{ass:sourcecondstrenght} hold. Then~$\mu^\dagger$ is the unique solution  of~\eqref{def:minprob}.
\end{proposition}
As a consequence, Proposition~\ref{prop:convqual} implies~$\bar{\mu} \rightharpoonup^* \mu^\dagger $. Moreover, according to~\cite[Proposition~1]{duval_peyre_2014}, the dual certificates~$\bar{\eta}$ associated to~\eqref{eq:estproblemmeasure0_bis2} approximate the minimal norm dual certificate~$\eta^\dagger$ in a suitable sense. Taking into account Assumption~\ref{ass:sourcecond} as well as Proposition~\ref{prop:first_order}, we thus conclude that the reconstruction of~$\mu^\dagger$ from~\eqref{def:estproblempoints} is governed by the convergence of the global extrema of~$\bar{\eta}$ towards those of~$\eta^\dagger$. However, in order to capitalize on this observation in our analysis, we need to compute a closed form expression for~$\eta^\dagger$. In general, this is intractable due to the global constraint $|\eta^{\dagger}(z)| \le 1$,~$z \in \Omega_s$. As a remedy, the authors of \cite{duval_peyre_2014} introduce a simpler proxy replacing this constraint by finitely many linear ones noting that
\[
\nabla \eta^\dagger(y^\dagger_n)=0, \quad \eta^\dagger(y^\dagger_n)= \sign(q^\dagger_n) \quad \text{for all } \quad n = 1,\ldots, N_s^{\dagger}.
\]
The computation of the associated vanishing derivative pre-certificate $\eta_{\PC} := K^* \Sigma_0^{-1} \zeta_{\PC} \in \mathcal{C}^2(\Omega_s)$ where
\begin{equation} \label{eq:defprecertif}
    \zeta_{\PC} = \argmin_{\zeta \in \R^{N_o}} \{ \norm{\zeta}_2:   \nabla \eta_{\PC}(y^\dagger_i) = 0, \quad \eta_{\PC}(y^\dagger_n)  = \operatorname{sign}(q^\dagger_n) \quad \text{ for all } \quad n = 1,\ldots, N_s^{\dagger} \}
\end{equation}
only requires the solution of a linear systems of equations and coincides with~$\eta^\dagger$ under appropriate conditions, see~\cite[Proposition 7]{duval_peyre_2014}.
Finally, in order to derive quantitative statements on the reconstruction error between~$\bar{\mu}$ and~$\mu^\dagger$, we require the non-degeneracy of the minimal norm dual certificate of~$\mu^\dagger$ in the sense of~\cite{duval_peyre_2014}. Since we aim to use~\eqref{eq:mainintro} in the context of optimal sensor placement, that is, we need to track the dependence of the involved constants on the measurement setting, we utilize the following quantitative definition; cf.~\cite{poon_2018}. 

\begin{definition} \label{def:nondegeneracy}We say that $\eta \in \mathcal{C}^2(\Omega_s)$ is $\theta-$non-degenerate or $\theta-$admissible for the sparse measure $\mu = \sum_{n=1}^{N_s} q_n \delta_{y_n}$ and~$\theta \in (0,1]$ if there holds
\begin{equation}\label{eq:nondegenerate_1}
\abs{\eta(y)} \leq 1 - \theta \min\left\{\theta,\, \min_{n=1,\ldots,N_s} \norm{w^\dagger_n(y-y_n)}_{2}^2 \right\}, \quad \eta(y_n)=\sign(q_n) \quad  
\quad\text{for all } \quad y\in \Omega_s
\end{equation}
and weights~$w^\dagger_n =\sqrt{|q_n^{\dagger}|}$.
\end{definition}
Due to the regularity of~$\eta$ one readily verifies that~\eqref{eq:nondegenerate_1} is equivalent to
\begin{equation}
\label{eq:coercive_Hess_eta}
-\sign{\eta}(y_n) \nabla^2 \eta(y_n) \geq 2\theta |w_n^{\dagger}|^2\operatorname{Id} \quad \text{for every} \quad n = 1,2,\ldots, N_s,
\end{equation}
as well as
\begin{align}
\label{eq:upper_bd_ball_eta}
\abs{\eta(y)} \leq 1-\theta^2, \quad \text{for all} \quad y \in \Omega_s \setminus \bigcup_{n=1,\ldots,N_s} B_{w_n^{\dagger}}(y_n, \sqrt{\theta}).
\end{align}

\section{Distances on spaces of measures}
\label{sec:assessing_reconstruction_errors}

In order to quantitatively study the reconstruction error of estimators of the source $\mu^{\dagger}$, we introduce a distance function on~$\mathcal{M}(\Omega_s)$ which measures the error between the estimated source measure~$\widehat{\mu}$ and the reference measure~$\miu^\dagger$. An obvious choice of distance would be the total variation norm on $\M(\Omega_s)$, however it is not suitable for quantifying the reconstruction error. In fact, evaluating~$d_{\TV}(\miu_1,\miu_2) = \norm{\miu_1 - \miu_2}_{\M(\Omega_s)}$ for sparse measures~$\miu_1, \miu_2 \in\M(\Omega_s)$ is simple by noting that
\begin{align*}
d_{\TV}(q_1\delta_{y_1},q_2\delta_{y_1})=|q_1-q_2|,
\end{align*}
but for $y_1 \neq y_2$, one has
\begin{align*}
d_{\TV}(q_1\delta_{y_1},q_2\delta_{y_2})=|q_1|+|q_2|,
\end{align*} 
that is,~$d_{\TV}$ does not quantify the reconstruction error of the source positions, and small perturbations of the source points lead to a constant error in the metric. Hence, in general one cannot rely on TV distance to evaluate the quality of the reconstruction. In the following, we consider an extension of the Hellinger-Kantorovich (H-K) metric~\cite{LieroMielkeSavare:2018} to signed measures, which possesses certain properties that will be discussed below. The construction of the H-K distance is more involved than another often used candidate, namely the Kantorovich-Rubinstein (K-R) distance (see, e.g.~\cite{PiccoliRossi:2016,lellmann_lorenz_schonlieb_valkonen_2014}) or flat metric, which is directly obtained as a dual norm of a space of Lipschitz functions (see Appendix~\ref{app:distance_metric}). It induces the same topology of weak* convergence, and is bounded by the H-K metric~\cite{LieroMielkeSavare:2018}. Since our estimates are going to be asymptotically sharp in H-K, but only an upper bound in K-R, we focus on H-K in the following.

The Hellinger-Kantorovich metric~\cite{LieroMielkeSavare:2018} is a generalization of the Wasserstein-\(2\) distance (see, e.g., \cite{MAL-073}) for measures which are not necessarily of the same norm. We first assume the case of positive measures \(\mu_1,\mu_2 \geq 0\) and define the H-K metric in terms of the Wasserstein-\(2\) metric as:
\[
d_{\HK}(\miu_1,\miu_2)^2
:= \inf
 \left\{W_2(\widetilde{\miu}_1,\widetilde{\miu}_2) \;\big|\;
 \widetilde{\miu}_1,\widetilde{\miu}_2 \in \mathcal{P}_2(\R^+\times\Omega_s) \colon h_2(\widetilde{\miu}_1) = \miu_1, h_2(\widetilde{\miu}_2) = \miu_2\right\}.
\]
Here, \(\mathcal{P}_2(\R^+\times\Omega_s) \) are the probability measures of with finite second moment on \(\R^+\times\Omega_s\), the two-homogeneous marginal is
\[
h_2(\widetilde{\miu}) = \int_{\R^+} r^2 \de \widetilde{\miu}(r,\cdot) \in \M(\Omega_s),
\]
and \(\R^+\times\Omega_s\) is endowed with a conic metric
\begin{equation}\label{eq:conicmetric}
d_{\text{cone}}((r_1,y_1), (r_2,y_2))^2
:= \left(\sqrt{r_1} - \sqrt{r_2}\right)^2 + 4 \sqrt{r_1 r_2}\sin_+^2(\norm{y_1-y_2}_2/2),
\end{equation}
where \(\sin_+(z) := \sin(\min\{\,z,\pi/2\,\})\).
For a detailed study of this metric and its properties as well as equivalent formulations in terms of Entropy-Transport problems we refer to~\cite{LieroMielkeSavare:2018}.

For signed measures, we note that for any distance based on a norm (such as the TV or K-R distance) one observes that
\begin{align}
\label{eq:signed_measure_trick_TV}
d(\miu_1, \miu_2) = \norm*{(\miu_1^+ + \miu_2^-) - (\miu_2^+ + \miu_1^-)}
= d(\miu_1^+ + \miu_2^-, \miu_2^+ + \miu_1^-),
\end{align}
by using the Jordan decomposition \(\miu_i = \miu^+_i - \miu^-_i\).
Motivated by \eqref{eq:signed_measure_trick_TV}, we define
\begin{equation}\label{eq:HK_decompostion}
d_{\HK} (\mu_1, \mu_2)
:= d_{\HK}(\mu_1^+ + \mu_2^-, \mu_2^+ + \mu_1^-),    
\end{equation}
which is indeed a metric on $\M(\Omega_s)$ and fulfills \(d_{\HK} (\mu_1, \mu_2) \leq d_{\HK}(\mu_1^+,\mu_2^+) + d_{\HK}(\mu_1^-, \mu_2^-)\). 

In contrast to the total variation distance, the Hellinger-Kantorovich distance between two Dirac measures $q_1\delta_{y_1}$ and $q_2\delta_{y_2}$ can be computed by
\[
d_{\HK}(q_1\delta_{y_1},q_2\delta_{y_2})^2 = (\sqrt{|q_1|}-\sqrt{|q_2|})^2+ 4 \sqrt{|q_1||q_2|} \sin_+^2(\norm{y_1-y_2}_2/2),
\]
which is exactly the conic metric given in \eqref{eq:conicmetric}.
Clearly, it is evidence that for small perturbations of both the source positions and coefficients, the resulting change of the H-K distance remains small. Hence, it is reasonable to employ this type of distance to measure the reconstruction error.

One next advantage of the H-K distance is that it is compatible with the weak* topology on $\M(\Omega_s)$, namely it induced weak* convergence on bounded set in $\M(\Omega_s)$.
\begin{proposition} The Hellinger-Kantorovich distance of signed measures defined in \eqref{eq:HK_decompostion} metrizes weak* convergence of signed measures on bounded set in $\M(\Omega_s)$. More precisely, a bounded sequence $\{\mu_n\}_{n \in \mathbb{N}} \subset \M(\Omega_s)$ converges weakly* to a measure $\mu$ if only if $d_{\HK}(\mu_n ,\mu) \to 0$ as $n \to \infty$.
\end{proposition}

\begin{proof}Assume that $d_{\HK}(\mu_n ,\mu) \to 0$ as $n \to \infty$. One can write
\begin{equation}\label{eq:distance_HK}
\mu_n - \mu = (\mu_n^+ + \mu^-) - (\mu^+ + \mu_n^-) =: \mu_n^1 - \mu_n^2,
\end{equation}
which implies
$d_{\HK}(\mu^1_n, \mu^2_n) = d_{\HK}(\mu_n, \mu) \to 0.$
Since \(\norm{\mu_n^i}_{\M} \leq \norm{\mu^{\pm}_n}_{\M} + \norm{\mu^{\mp}}_{\M} \leq 2M\) and the \(\HK\)-distance metrizes weak\(*\) convergence on bounded sequences of non-negative measures (see~\cite[Theorem~7.15]{LieroMielkeSavare:2018}), we have $\mu_n^1 - \mu_n^2 \rightharpoonup^* 0$, which means that $\mu_n \rightharpoonup^* \mu$.

Conversely, assume that $\mu_n \rightharpoonup^* \mu$. Consider the decomposition \eqref{eq:distance_HK} and suppose that the distance $d_{\HK}(\mu_n,\mu)$ does not converges to zero. Then there exists a subsequence, denoted by the same symbol, such that
\begin{align}
\label{eq:distance_positive}
d_{\HK}(\mu_n^1, \mu_n^2) = d_{\HK}(\mu_n, \mu)\geq \delta > 0.
\end{align}
We now use the fact that \(\norm{\mu_n^i}_{\M} \leq 2M\) to extract a further subsequence (again with the same symbol) such that \(\mu_n^i \rightharpoonup^* \widehat{\mu}^{i}\), which implies \(\mu_n - \mu = \mu^1_n - \mu^2_n \rightharpoonup^* \widehat{\mu}^{1} - \widehat{\mu}^{2}\).
 Due to~\eqref{eq:distance_positive} and the fact that the \(\HK\)-distance metrizes weak* convergence on bounded sequences of non-negative measures we have that \(\widehat{\mu}^{1} \neq \widehat{\mu}^{2}\) and thus \(\mu_n - \mu \rightharpoonup^* \widehat{\mu}^{1} - \widehat{\mu}^{2} \neq 0\). Thus the subsequence \(\{\mu_n\}_{n \in \mathbb{N}}\) does not converge weak* to \(\mu\) and the original sequence \(\{\mu_n\}_{n \in \mathbb{N}}\) can not converge to \(\mu\).
\end{proof}

To evaluate the reconstruction error, the distance between finitely supported measures is needed since the reference measure as well as the reconstructed measure are known to be sparse.  In fact, we only need a (sharp) upper bound for the H-K distance, which will be provided for the finitely supported case below in term of a (weighted) \(\ell^2\)-type distance.  This is yet another advantage of the H-K distance in comparison to other distances. 
\begin{proposition} \label{prop:HK_distance_estimate}
Let \(\miu\) and \(\miu^\dagger\) be finitely supported with the same number \(N\) of support points and \(\sign q_n = \sign q^\dagger_n\), for all $n = 1, \ldots, N$. Then we have
\begin{align*}
d_{\HK}(\miu,\miu^\dagger)^2
&\leq R(\bld{q},\bld{q}^{\dagger})
\sum_{n=1}^{N}\left(
  \frac{\abs{q_n - q_n^\dagger}^2}{4\abs{q_n^\dagger}}
+ \abs{q_n^\dagger} \, \norm{y_n - y_n^\dagger}^2_2
\right),
\end{align*}
where 
\(R(\bld{q},\bld{q}^\dagger)
:= \max\left\{\,\sqrt{{\abs{q_n}}/\abs{q_n^\dagger}},\, \sqrt{{\abs{q_n^\dagger}}/{{\abs{q_n}}}} \;:\; n=1,\ldots,N\,\right\}\).
\end{proposition}
Loosely speaking, the H-K distance between two discrete measures $\mu$ and $\mu^{\dagger}$ with the same number of support points could be upper bounded by a weighted $\ell^2$-type distance of their corresponding coefficients and positions.
\begin{proof}
We use that any finitely supported positive measure with \(N\) support points \(\miu\) can be extended with \(h_2(\widetilde{\miu}) = \miu\) according to 
\begin{align*}
\widetilde{\miu} =
\frac{1}{N}\sum_{n=1}^{N} \delta_{(r_n ,y_n)},
\quad\text{where } r_n = \sqrt{N \abs{q_n}}.
\end{align*}
In addition, notice that $d_{\HK}(\miu,\miu^\dagger) = d_{\HK} (\miu^1, \miu^2)$ where $\miu^1 := \miu^+ + \miu^{\dagger,-}$ and $\miu^2 := \mu^{\dagger,+} + \miu^-$ are positive measures with $N$ support of points. Thus, combining this with the fact that \((1/N) \, d_{\text{cone}}((r_1,y_1), (r_2,y_2))
= d_{\text{cone}}((r_1/\sqrt{N},y_1), (r_2/\sqrt{N},y_2))\) it follows:
\begin{align*}
d_{\HK}(\miu,\miu^\dagger)^2
&\leq \sum_{n=1}^{N_\dagger}\left[
\left(\sqrt{\abs{q_n}} - \sqrt{\abs{q_n^\dagger}}\right)^2
+ 4\sqrt{\abs{q_n}\abs{q_n^\dagger}} \cdot \sin_+^2\left(\norm{y_n - y_n^\dagger}_2/2\right)
\right] \\
&\leq \sum_{n=1}^{N_\dagger}\left(\frac{\left(q_n - q_n^\dagger\right)^2}{4\sqrt{\abs{q_n}\abs{q_n^\dagger}}}
+ \sqrt{\abs{q_n}\abs{q_n^\dagger}} \cdot \norm{y_n - y_n^\dagger}_2^2\right).
\end{align*}
Here, we have used \(\sin_+^2(\cdot) \leq (\cdot)^2\) and \((\sqrt{a} - \sqrt{b})^2 = (a-b)^2/(\sqrt{a}+\sqrt{b})^2 \leq (a-b)^2/(4\sqrt{ab})\). This immediately implies the estimate. 
\end{proof}

The previous result motivates to define a weighted \(\ell^2\)-norm for the given parameters $(\bld{q}; \bld{y}) \in  (\R\setminus\{0\})^{N} \times \Omega_s^N$.
More precisely, we define the weight \(w = \sqrt{\abs{\bld{q}}} := (\sqrt{|q_1|}, \cdots, \sqrt{|q_N|})\in(\R \setminus \{0\})^N\) and the associated weighted norm for a perturbation $(\delta\bld{q}; \delta\bld{y}) \in \R^{N} \times \R^{dN}$ as
\begin{align}
\label{eq:weighted_norm}
\norm{(\delta\bld{q};\delta\bld{y})}^2_W
:= \frac{1}{4}\norm{w^{-1}\delta \bld{q}}_2^2 + \norm{w \, \delta \bld{y}}_2^2
= \sum_{n=1}^{N} \left(\frac{|\delta q_n|^2}{4\abs*{q_n}}
+ \abs*{q_n} \norm{\delta y_n}_2^2 \right),
\end{align}
where \((w\delta\bld{y})_n = w_n \delta y_n\) denotes the entry-wise (Hadamard) product.
Here, the diagonal matrix \(W = \diag((w^{-2}/4; w^2; \ldots; w^2))\) induces the norm in \eqref{eq:weighted_norm}.
Then by Proposition~\ref{prop:HK_distance_estimate}, we have
\[
d_{\HK}(\miu,\miu^\dagger)^2
\leq R(\bld{q},\bld{q}^\dagger) \norm{(\bld{q}-\bld{q}^\dagger;\bld{y}-\bld{y}^\dagger)}^2_{W_\dagger},
\]
where \(W_\dagger\) is the diagonal weight matrix defined above for the weight \(w^\dagger = \sqrt{\abs{\bld{q}^\dagger}}\).
Moreover, two different weighted norms are equivalent up to the same factor
\begin{equation}\label{eq:equi_norm}
R(\bld{q},\bld{q}^\dagger)^{-1} \norm{(\delta\bld{q};\delta\bld{y})}^2_{W_\dagger} \le \norm{(\delta\bld{q};\delta\bld{y})}^2_W \leq R(\bld{q},\bld{q}^\dagger) \norm{(\delta\bld{q};\delta\bld{y})}^2_{W_\dagger}
\end{equation}
because \(R(\bld{q},\bld{q}^\dagger) = \max\{\,\norm{w/w^\dagger}_\infty,\,\norm{w^\dagger/w}_\infty\}\).
For \(\mu \approx \mu^\dagger\) the factor \(R(\bld{q},\bld{q}^\dagger)\) is arbitrarily close to one. In other words, asymptotically for \(\miu \approx \miu^\dagger\) the upper bound from Proposition~\ref{prop:HK_distance_estimate} is sharp:
\[
d_{\HK}(\miu, \miu^\dagger)^2
\approx \norm{(\bld{q} - \bld{q}^\dagger; \bld{y}-\bld{y}^\dagger)}^2_{W_\dagger}.
\]

\section{Fully explicit estimates for the deterministic reconstruction error} 
\label{sec:sensorplacement}


The Hellinger-Kantorovich distance allows us to quantify the reconstruction error between the unknown source~$\mu^\dagger $ and measures obtained by solving \eqref{eq:estproblemmeasure0_bis2}. This will be done in two steps. First, we study the approximation of~$\bld{m}^\dagger= (\bld{q}^\dagger; \bld{y}^{\dagger})$, i.e., the support points and coefficients of the ground truth, by stationary points~$\widehat{\bld{m}}=\widehat{\bld{m}}(\epsilon)$ of the nonconvex parametrized problem
\begin{equation}
\label{eq:estproblempoints_fixedN}
\min_{\bld{m} = (\bld{q}; \bld{y}) \in (\R \times \Omega_s)^{N_s}} 
\left[ \frac{1}{2} \norm{G(\bld{m}) - G(\bld{m}^{\dagger}) - \epsilon}_{\Sigma_0^{-1}}^2 + \beta \norm{\bld{q}}_{1} \right],
\end{equation}
where the source-to-observable map~$G$ satisfies
\begin{equation} \label{eq:paramtoobs}
G(\bld{m}) = G(\bld{q};\bld{y}) = k[\bld{x},\bld{y}]\bld{q} = \sum_{n=1}^N q_n k[\bld{x}, y_n].
\end{equation}
By Assumption~\ref{ass:kernel}, the latter is three times differentiable.
Notice that \eqref{eq:estproblempoints_fixedN} is obtained from \eqref{def:estproblempoints} by fixing $N_s = N_s^{\dagger}$ points of sources in the formulation. Hence, solutions, let alone stationary points, of problem~\eqref{eq:estproblempoints_fixedN} do not parametrize minimizers of~\eqref{eq:estproblemmeasure0_bis2} in general. Moreover, it is clear that problem~\eqref{eq:estproblempoints_fixedN} is primarily of theoretical interest since its practical realization requires knowledge of~$N^\dagger_s$.
Thus, in a second step, we investigate for which noises~$\epsilon$,~$\widehat{\bld{m}}$ parametrizes the unique solution of~\eqref{eq:estproblemmeasure0_bis2}.  
While these results build upon similar techniques as~\cite{duval_peyre_2014}, we give a precise, quantitative characterization of this asymptotic regime and clarify the dependence of the involved constants on the problem parameters, e.g., the measurement points~$\bld{x}$. This is necessary, for both, lifting these deterministic results to the stochastic setting in Section~\ref{sec:sparseinverse} as well utilizing the derived error estimates in the context of optimal sensor placement.  However, since these are merely intermediate steps in the derivation of our main result, we omit a detailed exposition at this point and direct the interested reader to Appendix~\ref{app:complicatedproofs}. In the following, a central role will be played by the linearized problem
\begin{align} \label{def:linearizedpoints}
\min_{\delta \bld{m} =(\delta \bld{q}; \delta\bld{y}) \in \mathbb{R}^{(1+d)N_s}}
\left \lbrack \dfrac{1}{2} \norm{G'(\bld{m}^\dagger)\delta \bld{m} - \epsilon}_{\Prec}^2 + \beta \operatorname{sign}(\bld{q}^\dagger)^\top \delta \bld{q} \right\rbrack.
\end{align}
Note that here we have linearized both, the mapping~$G$ as
\begin{align*}
G(\bld{q}^\dagger+\delta\bld{q}, \bld{y}^\dagger+\delta\bld{y})
&\approx
G(\bld{q}^\dagger,\bld{y}^\dagger)+
G' (\bld{q}^\dagger,\bld{y}^\dagger)(\delta \bld{q},\delta \bld{y}) \\
&= \ke[\bld{x},\bld{y}^\dagger]\bld{q}^\dagger
+ \ke[\bld{x},\bld{y}^\dagger]\delta\bld{q} + (\nabla_y^\top \ke [\bld{x},\bld{y}^\dagger] \circ \bld{q}^\dagger) \delta \bld{y},
\end{align*}
using that
\[
G'(\bld{m}) = \left(
\ke[\bld{x},\bld{y}] 
\quad \nabla_y^\top \ke [\bld{x},\bld{y}] \circ \bld{q} \right) \quad \text{where} \quad (\nabla_y^\top \ke [\bld{x},\bld{y}^\dagger] \circ \bld{q}^\dagger)_{i,j}
:= \nabla_y \ke (x_i,y^\dagger_j)^\top q^\dagger_j,
\]
as well as the~$\norm{\cdot}_{1}-$norm with
\[
\norm{\bld{q}^\dagger + \delta\bld{q}}_{1}
\approx
\norm{\bld{q}^\dagger}_{1} +
\operatorname{sign}(\bld{q}^\dagger)^\top \delta \bld{q}.
\]

The following proposition characterizes the solutions of~\eqref{eq:estproblempoints_fixedN} and~\eqref{def:linearizedpoints}. Since its proof relies on standard computations, we omit it for the sake of brevity.
\begin{proposition}

The solutions \(\bar{\bld{m}}\) to~\eqref{eq:estproblempoints_fixedN} fulfill the stationarity condition
\begin{equation}
\label{eq:stationarity_R}
S(\bar{\bld{m}})
:=
G'(\bar{\bld{m}})^\top \Sigma_0^{-1}(G(\bar{\bld{m}}) - G(\bld{m}^\dagger) - \epsilon) + \beta(\bar{\bld{\rho}};\bld{0}) = 0,
\end{equation}
for some \(\bar{\bld{\rho}} \in \partial\norm{\bar{\bld{q}}}_1\). The solutions of~\eqref{def:linearizedpoints} satisfy
\[
G'(\bld{m}^\dagger)^\top \Sigma_0^{-1}(G'(\bld{m}^\dagger)\delta \widehat{\bld{m}} - \epsilon) + \beta \vectorsgn = 0,
\]
where \( \bld{\rho} = \sign \bld{q}^\dagger\). If \(G'(\bld{m}^\dagger)\) has full column rank then the Fisher information matrix
\begin{equation}\label{eq:fisher}
\mathcal{I}_0 := G'(\bld{m}^\dagger)^\top \Sigma_0^{-1} G'(\bld{m}^\dagger)
\end{equation}
is invertible and the unique solution of~\eqref{def:linearizedpoints} is given by
\begin{equation}
\label{eq:delta_z_hat}
    \begin{aligned}
    \delta \widehat{\bld{m}} (\epsilon)&:= \mathcal{I}_0^{-1} \left( G'(\bld{m}^\dagger)^\top \Sigma_0^{-1} \epsilon - \beta \vectorsgn \right) \\
    &= (\Sigma_0^{-1/2} G'(\bld{m}^\dagger))^+ \Sigma_0^{-1/2} \epsilon - \beta \mathcal{I}_0^{-1}\vectorsgn 
\end{aligned}
\end{equation}
where \((\Sigma_0^{-1/2} G'(\bld{m}^\dagger))^+ = \mathcal{I}_0^{-1} G'(\bld{m}^\dagger)^\top \Sigma_0^{-1/2}\) is the pseudo-inverse of \(\Sigma_0^{-1/2} G'(\bld{m}^\dagger)\).
\end{proposition}
Since~\eqref{eq:estproblempoints_fixedN} is nonconvex,~the stationarity condition~\eqref{eq:stationarity_R} is only necessary but not sufficient for optimality. In the following, we call any solution to~\eqref{eq:stationarity_R} a stationary point.


\subsection{Error estimates for stationary points}
In this section, we show that for sufficiently small noise~$\epsilon$, problem~\eqref{eq:estproblempoints_fixedN} admits a unique stationary point~$\widehat{\bld{m}}(\epsilon)$ in the vicinity of~$\bld{m}^\dagger$. Moreover, loosely speaking,~$\bld{m}^\dagger$ and~$\bld{m}^\dagger+\delta \widehat{\bld{m}}(\epsilon)$ provide Taylor expansions of zeroth and first order, respectively, for~$\widehat{\bld{m}}(\epsilon)$.
\begin{proposition}\label{prop:pertub}
Suppose that \(G'(\bld{m}^\dagger)\) has full column rank. Then, for some constant $\Ca = \Ca(\parameter)$ and radius~$\hat{r}>0$ and all~$\epsilon$ with~$\Ca(\norm{\epsilon}_{\Sigma^{-1}_0} + \beta) \le 1$, the stationarity condition~\eqref{eq:estproblempoints_fixedN} admits a unique solution $\widehat{\bld{m}}=\widehat{\bld{m}}(\epsilon)$ on ${B_{W^{\dagger}}}(\bld{m}^{\dagger},(3/2)\hat{r})$. Moreover, the stationary point satisfies $\widehat{\bld{m}} \in {B_{W^{\dagger}}}(\bld{m}^{\dagger},\hat{r})$ as well as
\begin{align*}
\norm{\widehat{\bld{m}} - \bld{m}^\dagger}_{W_\dagger} &\le 2 \norm{\delta \widehat{\bld{m}}}_{W_\dagger} \le \Ca (\norm{\epsilon}_{\Prec} + \beta), 
\\
\norm{\widehat{\bld{m}} - \bld{m}^\dagger - \delta\widehat{\bld{m}}}_{W_\dagger} &\leq \Ca^2 (\norm{\epsilon}_{\Prec} + \beta)^2. 
\end{align*}
\end{proposition}

For the sake of brevity, we omit by now the proof of Proposition~\ref{prop:pertub}, which is then presented in Appendix~\ref{app:complicatedproofs}.

\begin{remark}

We note that~$\Ca$ depends monotonically on the norm of the inverse Fisher information matrix; see Remark~\ref{rem:constant_C1}.
Moreover, the dependency on the ground truth \(\mu^\dagger\) is only in terms of the norm \(\norm{\bld{q}^\dagger}_{1}\), and distances of \(y_n^\dagger\) to the boundary and \(q_n^\dagger\) to zero.
 \end{remark}

\subsection{Error estimates for reconstructions of the ground truth}
As mentioned in the preceding section, solving the stationarity equation \eqref{eq:stationarity_R} for~$ \widehat{ \bld{m}}=(\widehat{\bld{y}},\widehat{\bld{q}})$ is not feasible in practice since it presupposes knowledge of~$N^\dagger_s$. Moreover, recalling that~$\widehat{\bld{m}}$ is merely a stationary point, the parametrized measure   
\begin{align} \label{def:measparamm}
    \widehat{\mu}=\sum^{N^\dagger_s}_{n=1} \widehat{q}_n \delta_{\widehat{y}_n}
\end{align}
is not necessarily a minimizer of~\eqref{eq:estproblemmeasure0_bis2}. In this section, our primary goal is to show that~$\widehat{\bld{m}}$ indeed parametrizes the unique solution of problem~\eqref{eq:estproblemmeasure0_bis2} if the minimum norm dual certificate~$\eta^\dagger$ associated to~\eqref{def:minprob} is~$\theta$-admissible and if the set of admissible noises~$\epsilon$ is further restricted. A fully-explicit estimate for the reconstruction error between~$\widehat{\mu}$ and the ground truth~$\mu^\dagger$ in the Hellinger-Kantorovich distance then follows immediately.   
For this purpose, recall from \cite[Proposition 7]{duval_peyre_2014} that the non-degeneracy of~$\eta^\dagger$ implies
\begin{equation}\label{eq:precertificate}
\eta^\dagger= \eta_{\text{PC}} = \Ke^* \Sigma^{-1/2}_0 (G'(\bld{m}^\dagger) \Sigma^{-1/2}_0)^+ \vectorsgn
= \Ke^* \Prec G'(\bld{m}^\dagger)\mathcal{I}_0^{-1}\vectorsgn.
\end{equation}
where~$\eta_{\text{PC}}$ denotes the vanishing derivative pre-certificate from Section~\ref{subsec:minnorm}.

We first prove that
\begin{align*}
    \widehat{\eta}=\beta^{-1}K^*\Prec(z^d(\epsilon)-K\widehat{\mu})=\beta^{-1}K^*\Prec(G(\bld{m}^\dagger)+\epsilon-G(\widehat{\bld{m}})) 
\end{align*}
is~$\theta/2$-admissible for certain~$\epsilon$ and $\beta$. 

\begin{proposition}\label{prop:admissibility} Let the assumptions in Proposition~\ref{prop:pertub} be satisfied and $\eta^\dagger$ be $\theta-$admissible for $\mu^{\dagger}$, $\theta \in (0,1]$. Then there exists a constant $\Cb = \Cb(\parameter)$ such that if 
\begin{align}
\label{eq:set_A1_tilde} &\Ca(\norm{\epsilon}_{\Sigma^{-1}_0} + \beta) \le \sqrt{\theta/32}, \\
\label{eq:set_A2} &\Cb  \beta^{-1}\left((\norm{\epsilon}_{\Sigma^{-1}_0} + \beta)^2 + \norm{\epsilon}_{\Sigma_0^{-1}}\right) \le \theta^2/32, 
\end{align}
then the function~$\widehat{\eta}$ is $\theta/2-$admissible for~$\widehat{\miu}$.
\end{proposition}

The proof of Proposition~\ref{prop:admissibility} is then provided in Appendix~\ref{app:complicatedproofs}.

\begin{remark}
We note that $\Cb$ depends monotonically on the norm of the inverse Fisher information matrix; see Remark~\ref{rem:constant_C2}. Moreover, the dependency on the ground truth \(\mu^\dagger\) is only in terms of the norm \(\norm{\bld{q}^\dagger}_{1}\), and distances of \(y_n^\dagger\) to the boundary and \(q_n^\dagger\) to zero.
\end{remark}

As a consequence, we conclude that the solution to~\eqref{eq:estproblemmeasure0_bis2} is unique and parametrized by~$\widehat{\bld{m}}$. Moreover, its~H-K distance to~$\mu^\dagger$ can be bounded in terms of the linearization~$\delta \widehat{\bld{m}}$.
\begin{theorem} \label{thm:uniqueandest}
Let the assumptions of Proposition~\ref{prop:admissibility} hold. Then the solution of~$\eqref{eq:estproblemmeasure0_bis2}$ is unique and given by~$\widehat{\mu}$ from~\eqref{def:measparamm}.
There holds
\begin{align} \label{eq:estforoptHK}
    d_{\HK}(\widehat \miu,\miu^\dagger)^2 \leq 8  \norm{\delta \widehat{\bld{m}}}^2_{W_\dagger}.
\end{align}
\end{theorem}   
\begin{proof}
From Proposition~\ref{prop:admissibility}, we conclude that~$\widehat{\eta}$ is~$\theta/2$-admissible for~$\widehat{\mu}$. Consequently, we have~$\widehat{\eta} \in \partial \|\widehat{\miu}\|_{\mathcal{M}(\Omega_s)}$, i.e.,~$\widehat{\miu}$ is a solution of~\eqref{eq:estproblemmeasure0_bis2}. It remains to show its uniqueness. For this purpose, it suffices to argue that
\begin{align*}
    K_{| \supp \widehat{\mu} }=k\lbrack \bld{x}, \widehat{\bld{y}} \rbrack \in \R^{N_o \times N^\dagger_s}
\end{align*}
is injective, see, e.g., the proof of~\cite[Proposition 3.6]{pieper_walter_2021}. Assume that this is not the case. Then, following~\cite[Theorem B.4]{pieper_tang_trautmann_walter_2020}, there is~$\bld{v}\neq 0$ with~$k\lbrack \bld{x}, \bld{y} \rbrack \bld{v}=0$ and~$\tau \neq 0$ such that the measure \(\Tilde{\mu}\) parametrized by \(\Tilde{\bld{m}} = (\Tilde{\bld{q}};\widehat{\bld{y}})\) with \(\Tilde{\bld{q}} = \widehat{\bld{q}} + \tau\bld{v}\)
is also a solution of~\eqref{eq:estproblemmeasure0_bis2} (choose the sign of \(\tau\) to not increase the \(\ell_1\)-regularization, and the magnitude small not to change the sign of \(\tilde{\bld{q}}\)) and~$\Tilde{\bld{q}} \neq \widehat{\bld{q}}$. For~$s \in (0,1)$, set~$\bld{q}_s=(1-s) \widehat{\bld{q}} + s \Tilde{\bld{q}}$. By convexity of~\eqref{eq:estproblemmeasure0_bis2}, the measure parametrized by~$\bld{m}_s=( \bld{q}_s;\widehat{\bld{y}})$ is also a minimizer of~\eqref{eq:estproblemmeasure0_bis2}. Consequently,~$\bld{m}_s$ is a solution of~\eqref{eq:estproblempoints_fixedN} and thus also a stationary point. Finally, noting that~$\bld{m}_s \neq \widehat{\bld{m}} $,~$s \in (0,1)$, and~$\lim_{s\rightarrow 0} \bld{m}_s= \widehat{\bld{m}} $, we arrive at a contradiction to the uniqueness of stationary points in the vicinity of~$\bld{m}^\dagger$.
The estimate in~\eqref{eq:estforoptHK} immediately follows from
    \begin{align*}
        d_{\HK}(\widehat{\mu},\miu^\dagger)^2 \leq R(\widehat{\bld{q}},\bld{q}^\dagger ) \norm{\widehat{\bld{m}}-\bld{m}^\dagger}^2_{W_\dagger} \leq 2 \norm{\widehat{\bld{m}}-\bld{m}^\dagger}^2_{W_\dagger} \leq 8  \norm{\delta \widehat{\bld{m}}}^2_{W_\dagger}
        \qquad\qedhere
    \end{align*}
\end{proof}
\section{Inverse problems with random noise}
\label{sec:stochasticnoise} 
Finally, let~$(\mathcal{D},\mathcal{F},\mathbb{P})$ denote a probability space and consider the stochastic measurement model
\[
z^d(\varepsilon) = 
\Ke \miu^\dagger + \varepsilon,
\]
where the noise is distributed according to~$\varepsilon \sim \gamma_p = \mathcal{N}(0,p^{-1}\Sigma_0)$ for some $p>0$ representing the overall precision of the measurements. Mimicking the deterministic setting, we are interested in the reconstruction of the ground truth~$\mu^\dagger$ by solutions obtained from~\eqref{eq:estproblemmeasure0_bis2} for realizations of the random variable~$\eps$. By utilizing the quantitative analysis presented in the preceding section, we provide an upper bound on the worst-case mean-squared error 
\begin{align*}
    \mathbb{E}_{\gamma_p} \left \lbrack \sup_{ \mu \in \mathfrak{M}(\cdot)} d_{\text{HK}}(\mu,\mu^\dagger )^2\right \rbrack= \int_{\R^{N_o}} \sup_{ \mu \in \mathfrak{M}(\epsilon)} d_{\text{HK}}(\mu,\mu^\dagger )^2~\mathrm{d}\gamma_p(\epsilon)
\end{align*}
for a suitable a priori parameter choice rule~$\beta=\beta(p)$. Note that the expectation is well-defined according to Appendix~\ref{app:meas}.

 
\subsection{A priori parameter choice rule}
Before stating the main result of the manuscript, let us briefly motivate the particular choice of the misfit term in~\eqref{eq:estproblemmeasure0_bis2} as well as the employed parameter choice rule from the perspective of the stochastic noise model. Since we consider independent measurements, their covariance matrix~$\Sigma= p^{-1}\Sigma_0$ is diagonal with \(\Sigma_{jj} =  \sigma_j^2\) for variances \(\sigma^2_j > 0\), \(j=1,\ldots,N_o\). This corresponds to performing the individual measurements with independent sensors of variable precision \(p_j = 1/\sigma_j^2\). We call
\[
\ptot = \sum_{j=1}^{N_o} p_j = \sum_{n=1}^{N_o} \sigma_j^{-2} = \tr(\Sigma^{-1}).
\]
the total precision of the sensor array.
It can be seen that its reciprocal \(\sigma^2_{\text{tot}} = 1/\ptot\) corresponds to the harmonic average of the variances divided by the number of sensors \(N_o\). Therefore, the misfit in~\eqref{eq:estproblemmeasure0_bis2} satisfies
\[
\norm{K\mu - \zd(\epsilon)}^2_{\Prec} = \frac{1}{\ptot}\sum_{j=1}^{N_o} \sigma^{-2}_j \abs{{\lbrack K \mu \rbrack_j} - \zd(\epsilon)_j}^2.
\]
For identical sensors and measurements~$\varepsilon \sim \mathcal{N}(0, \operatorname{Id}_{N_o} )$ this simply leads to the scaled Euclidean norm \((1/N_o)\norm{K\mu - \zd(\epsilon)}^2_2\). In general, by increasing the total precision of the sensor setup \(\ptot\), we improve the measurements by proportionally decreasing the variances by \(\sigma^2_{\text{tot}}\). While this will decrease the expected level of noise through its distribution, it will not affect the misfit functional, which is just influenced by \(\Sigma_0\), or the normalized variances \(\sigma^2_{0,j} = \sigma^2_j/\sigma^2_{\text{tot}}\).

Moreover, since $\varepsilon \sim \mathcal{N}(0,\Sigma)$, we have $ \Sigma^{-1/2}\varepsilon \sim \mathcal{N}(0,\id_{N_o})$ and by direct calculations, the following estimate holds
\[
\frac{N_o}{\sqrt{N_o + 1}} \le  \E_{\gamma_p}[\norm{\varepsilon}_{\Sigma^{-1}}] \le \sqrt{N_o}.
\]
Hence, with high probability, realizations of the error fulfill the estimate 
\[\sqrt{\sum_{j=1}^{N_o} \epsilon_j^2/\sigma^2_j} = \norm{\epsilon}_{\Sigma^{-1}} = \sqrt{\ptot} \norm{\epsilon}_{\Sigma^{-1}_0} \leq C\sqrt{N_o}
\] 
and thus \(\norm{\epsilon}_{\Sigma^{-1}_0} \lesssim 1/\sqrt{\ptot}\).
Thus, we consider the expected noise \(\sigma_{\text{tot}} = 1/\sqrt{\ptot}\) as an (expected) upper bound for the noise. This motivates the parameter choice rule
\begin{equation*}
\beta(p) = \beta_0 / \sqrt{\ptot}= \beta_0 \tr (\Sigma^{-1})^{-1/2}
\end{equation*}
for some \(\beta_0>0\) large enough.

\subsection{Quantitative error estimates in the stochastic setting}
\label{sec:quantativeerror} 
We are now prepared to prove a quantitative estimate on the worst-case mean-squared error by lifting the deterministic result of Theorem~\ref{thm:uniqueandest} to the stochastic setting.
\begin{theorem}
\label{thm:MSE_bound}
Assume that \(\eta^\dagger\) is \(\theta\)-admissible for \(\theta \in (0,1)\) and set~$\beta(p)=\beta_0/\sqrt{p}$. Then there exists

\begin{align*}
    \overline{\ptot} = \beta_0^2 c_{\ptot}(
\theta, k, \mu^{\dagger}, \normI)
\end{align*}

such that for $p \ge \overline{\ptot}$, there holds
\begin{equation}\label{eq:estimate_maintheorem}
\E_{\gamma_p}\left[\sup_{{\miu} \in \mathfrak{M}(\cdot)} d_{\HK}({\miu},\miu^{\dagger})^2\right]
\le 8 \E_{\gamma_p}[\norm{\delta\widehat{\bld{m}}}^2_{W_\dagger}] +  \Cd\exp \left[-\left(\dfrac{
\theta^2 \beta_0}{64 \Ce}\right)^2/(2N_o)\right],
\end{equation}
where $\Cd 
 = 2\norm{\mu^\dagger}_{\M(\Omega_s)} + \sqrt{2N_o}/(2\beta_0 \sqrt{\ptot}) $ and~$\Ce= \max \{C_1,C_2 \}$. In addition, the expectation $\E_{\gamma_p}[\norm{\delta\widehat{\bld{m}}}^2_{W_\dagger}]$ has the closed form
 \begin{align} \label{eq:closedformlin}
         \E_{\gamma_p}[\norm{\delta\widehat{\bld{m}}}^2_{W_\dagger}]= \frac{1}{p} \left(  \tr(W_{\dagger}\mathcal{I}_0^{-1}) + \beta_0^2 \norm{\mathcal{I}_0^{-1} \vectorsgn}^2_{W_{\dagger}}\right).
 \end{align}
\end{theorem}
\begin{proof} Define the sets
\begin{align*}
A_1 = \left\{ \epsilon : \Ce \beta(p)^{-1}\norm{\epsilon}_{\Sigma_0^{-1}} \le \frac{\theta^2}{64}\right\}, \quad A_2 = \left\{ \epsilon : \Ce \beta(p)^{-1}(\norm{\epsilon}_{\Sigma^{-1}_0} + \beta(p))^2 \le \frac{\theta^2}{64}\right\}.
\end{align*}
By a case distinction, we readily verify
\begin{align*}
    \R^{N_o} \setminus (A_1 \cap A_2) \subset (\R^{N_o}\setminus A_1) \cup ((\R^{N_o}\setminus A_2) \cap A_1) 
\end{align*}
and thus
\begin{equation}\label{eq:bound2}
\begin{aligned}
\E_{\gamma_p}\left\lbrack \sup_{\mu \in \mathfrak{M}(\cdot)} d_{\HK}(\miu,\miu^{\dagger})^2 \right \rbrack 
\le &\int_{A_1 \cap A_2} \sup_{\mu \in \mathfrak{M}(\epsilon)} d_{\HK}(\miu,\miu^{\dagger})^2 \de \gamma_p(\epsilon)\\   
&+ \underbrace{\int_{\R^{N_o}\backslash A_1} \sup_{\mu \in \mathfrak{M}(\epsilon)} d_{\HK}({\miu},\miu^{\dagger})^2 \de \gamma_p(\epsilon)}_{I_1} \\
&+ \underbrace{\int_{(\R^{N_o}\backslash A_2) \cap A_1} \sup_{\mu \in \mathfrak{M}(\epsilon)} d_{\HK}(\miu,\miu^{\dagger})^2 \de \gamma_p(\epsilon)}_{I_2}.
\end{aligned}
\end{equation}

For $\epsilon \in A_1 \cap A_2$, we have
\begin{align*}
\Cb  \beta(p)^{-1} \left((\norm{\epsilon}_{\Sigma^{-1}_0} + \beta(p))^2  + \norm{\epsilon}_{\Sigma_0^{-1}}\right)  &\leq 
 \Ce   \beta^{-1}(p)(\norm{\epsilon}_{\Sigma^{-1}_0} + \beta(p))^2 + \Ce   \beta^{-1}(p)\norm{\epsilon}_{\Sigma_0^{-1}} \\& \leq \frac{\theta^2}{64}+\frac{\theta^2}{64}=\frac{\theta^2}{32},
\end{align*}
i.e.,~$\epsilon$ satisfies~\eqref{eq:set_A2}. Moreover, expanding the square in the definition of~$A_2$, we conclude that~\eqref{eq:set_A1_tilde} also holds due to
\begin{align*}
    2\Ce (\norm{\epsilon}_{\Sigma^{-1}_0} + \beta(p)) \leq  \frac{\theta^2}{32} \leq \frac{\sqrt{\theta}}{2 \sqrt{32}}.
\end{align*}
Hence, for~$\epsilon \in A_1 \cap A_2$, there holds~$\mathfrak{M}(\epsilon)=\{\widehat{\mu}\}$ and
\begin{equation}\label{eq:bound3}
\sup_{\mu \in \mathfrak{M}(\epsilon)} d_{\HK}(\miu,\miu^{\dagger})^2= d_{\HK}(\widehat{\miu},\miu^{\dagger})^2  \le 8 \norm{\delta{\widehat{\bld{m}}}}^2_{W_\dagger}
\end{equation}
by Proposition \ref{thm:uniqueandest}.
 Next, we estimate $I_1$ by
\begin{align*}
d_{\HK}(\widehat{\miu}, \miu^{\dagger})^2
\leq \norm{\miu^\dagger}_{\mathcal{M}(\Omega_s)} + \norm{\widehat{\miu}}_{\mathcal{M}(\Omega_s)}
\leq 2\norm{\miu^\dagger}_{\mathcal{M}(\Omega_s)} + \norm{\epsilon}_{\Sigma_0^{-1}}^2/(2\beta_0/\sqrt{\ptot})
\end{align*}
applying Proposition~\ref{prop:bound_on_barq} and \cite[Proposition 7.8]{LieroMielkeSavare:2018}. Together with Lemma \ref{lem:probabilitybound} this yields
\begin{equation}\label{eq:bound4}
\begin{aligned}
I_1 = \int_{\R^{N_o}\backslash A_1}d_{\HK}( \widehat{\miu}, \miu^{\dagger})^2 \de \gamma_p(\epsilon) &\le  \int_{\R^{N_o}\backslash A_1} \left(2\norm{\miu^\dagger}_{\mathcal{M}(\Omega_s)} + \norm{\epsilon}_{\Sigma^{-1}}^2/(2\beta_0\sqrt{\ptot})\right) \de\gamma_p(\epsilon) \\
&\le \left(2\norm{\mu^\dagger}_{\M(\Omega_s)} + \dfrac{\sqrt{2N_o}}{2\beta_0 \sqrt{\ptot}} \right) \exp \left[-\left(\dfrac{
\theta^2 \beta_0}{64 \Ce }\right)^2/(2N_o)\right].
\end{aligned}
\end{equation}
Finally, for $\epsilon \in (\R^{N_o}\backslash A_2) \cap A_1$, one has
\begin{align*}
 \ptot^{1/4} \left(\frac{\theta^2 \beta_0}{64\Ce}\right)^{1/2} - \beta_0 < \norm{\epsilon}_{\Sigma^{-1}}= p^{-1/2} \norm{\epsilon}_{\Sigma^{-1}_0} \le \dfrac{\theta^2 \beta_0}{64\Ce}
\end{align*}
where the first inequality follows from~$\epsilon \not \in A_2$ and the second follows from~$\varepsilon \in A_1$. 
Hence, if we choose 
\[
p 
\ge \beta_0^2\left(\frac{\theta^2}{64 \Ce} + 1\right)^4 \big/ \left(\frac{\theta^2 }{64 \Ce}\right)^2
:= \beta_0^2 \overline c_{\ptot}:= \overline{p} 
\]
then $(\R^{N_o}\backslash A_2) \cap A_1$ is empty and~$I_2=0$. Together with \eqref{eq:bound2}--\eqref{eq:bound4}, we obtain \eqref{eq:estimate_maintheorem} for every $\ptot \ge \overline{\ptot}$. The equality in~\eqref{eq:closedformlin} follows immediately from the closed form expression~\eqref{eq:delta_z_hat} for~$\delta \widehat{\bld{m}}$ and~$\varepsilon \sim \mathcal{N}(0, p^{-1}\Sigma_0)$.
\end{proof}

Let us interpret this result: By choosing $\beta_0$ large enough, the second term on the right hand side of \eqref{eq:estimate_maintheorem_pinfty} becomes negligible, i.e.,
\begin{equation}\label{eq:estimate_maintheorem_pinfty}
\E_{\gamma_p}\left[\sup_{{\miu} \in \mathfrak{M}(\cdot)} d_{\HK}({\miu},\miu^{\dagger})^2\right]
\le 8 \E_{\gamma_p}[\norm{\delta\widehat{\bld{m}}}^2_{W_\dagger}]+ \delta
\end{equation}
for some~$0 <\delta \ll 1$. As a consequence, due to its closed form representation~\eqref{eq:closedformlin}, $\E_{\gamma_p}[\norm{\delta\widehat{\bld{m}}}^2_{W_\dagger}]$ provides a computationally inexpensive, approximate upper surrogate for the worst-case mean-squared error which vanishes as~$p \rightarrow \infty$. Moreover, due to its explicit dependence on the measurement setup, it represents a suitable candidate for an optimal design criterion in the context of optimal sensor placement for the class of sparse inverse problems under consideration. This potential will be further investigated in a follow-up paper.  

\begin{remark}
It is worth mentioning that the constant $8$ appearing on the right hand side of \eqref{eq:estimate_maintheorem_pinfty} is not optimal and is primarily a result of the proof technique. In fact, by appropriately selecting constants in Propositions~\ref{prop:pertub_estimate} and \ref{prop:pertub}, it is possible to replace $8$ by $1 + \delta$, where $0 < \delta \ll 1$ at the cost of increasing~$\bar{p}$. We will illustrate the sharpness of the estimate of the worst-case mean-squared error by~$\E_{\gamma_p}[\norm{\delta\widehat{\bld{m}}}^2_{W_\dagger}]$ in the subsequent numerical results.
\end{remark}
\begin{remark}
Relying on similar arguments as in the proof of Theorem~\ref{thm:MSE_bound}, we are also able to derive pointwise estimates on the Hellinger-Kantorovich distance which hold 
with high probability. Indeed, noticing that \eqref{eq:bound3} holds in the set $A_1 \cap A_2$, we derive a lower probability bound for \(\P( \varepsilon \in A_1 \cap A_2)\) by noticing that
\begin{align*}
\P( \varepsilon \in A_1 \cap A_2) &\ge \P( \varepsilon \in A_1) + \P( \varepsilon \in A_2) - 1 \\ 
&\ge 1 - \P(\varepsilon \in \R^{N_o}\backslash A_1) - \P(\varepsilon \in     \R^{N_o}\backslash A_2).
\end{align*}

By invoking Lemma \ref{lem:probabilitybound}, one has
\begin{align*}
&\P(\varepsilon \in \R^{N_o}\backslash A_1) = \P \left(\norm{\varepsilon}_{\Sigma^{-1}} > \dfrac{\theta^2 \beta_0}{64 \Ce}\right) \le 2\exp\left[- \left(\dfrac{\theta^2 \beta_0}{64 \Ce}\right)^2/(2 N_o) \right],\\
&\P(\varepsilon \in \R^{N_o}\backslash A_2) = \P \left( \norm{\varepsilon}_{\Sigma^{-1}} > \dfrac{p^{1/2} \theta \beta_0^{1/2} }{8 \Ce^{1/2}} - \beta_0 \right) \le 2 \exp\left[- \left(\dfrac{p^{1/2} \theta \beta_0^{1/2} }{8 \Ce^{1/2}} - \beta_0\right)^2/(2 N_o) \right].
\end{align*}
Hence, since $\exp(-x^2) \to 0$ as $x \to \infty$, we can see that for every $\delta \in (0,1)$, one can choose $\beta_0$ and~$p$ large enough such that
\begin{align*}
\exp\left[- \left(\dfrac{\theta^2 \beta_0}{64 \Ce}\right)^2/(2 N_o) \right] < \delta/4,\quad \exp\left[- \left(\dfrac{p^{1/2} \theta \beta_0^{1/2} }{8 \Ce^{1/2}} - \beta_0\right)^2/(2 N_o) \right] < \delta/ 4,
\end{align*}
which implies $\P(\varepsilon \in A_1 \cap A_2) \ge 1 - \delta$. Therefore, with probability at least $1-\delta$, we have
\begin{equation*}
\sup_{{\miu} \in \mathfrak{M}(\epsilon)} d_{\HK}({\miu},\miu^{\dagger})^2 \le 8\norm{\delta \widehat{\bld{m}}}^2_{W_{\dagger}}
\end{equation*}
for realization~$\epsilon$ of the noise.
Furthermore, employing Lemma \ref{lem:probabilitybound} again, we know that with probability at least $1- \delta$, and independently from $p$, one has $\norm{\epsilon}_{\Sigma^{-1}} \le \sqrt{-2N_o\ln(\delta/2)}$. Hence, by Proposition \ref{prop:pertub} together with $\varepsilon \in A_1 \cap A_2$, we have
\begin{align*}
\sup_{{\miu} \in \mathfrak{M}(\epsilon)} d_{\HK}({\miu},\miu^{\dagger})^2 &\le 8\Ca p^{-1/2}(\norm{\epsilon}_{\Sigma^{-1}} + \beta_0) \\ &\le 8\Ca p^{-1/2} \left(\sqrt{-2N_o\ln(\delta/2)} + \beta_0 \right)
\end{align*}
with probability at least $1-2 \delta$.
\end{remark}

\section{Numerical results}\label{sec:numerics}
We end this paper with the study of some numerical examples to illustrate our theory. We consider a simplified version of Example~\ref{ex:2}:

\begin{itemize}
\item The source domain $\Omega_s$ and observation domain $\Omega_o$ are the interval $[-1, 1]$.

\item The reference measure is given by $\mu^{\dagger} = 0.4 \delta_{-0.7} + 0.3 \delta_{-0.3} - 0.2 \delta_{0.3} \in \M(\Omega_s)$.

\item
The kernel $k:[-1,1] \times [-1,1] \to \R$ is defined as
\[
k(x,y) = \exp \left(- \dfrac{(x-y)^2}{2\sigma^2} \right), \sigma = 0.2, \quad x,y \in [-1,1].
\]
\item The measurement points $\{x_1, \ldots, x_{N_o}\} \subset \Omega_o$ vary between the individual examples and are marked by grey points in the respective plots. The associated noise model is given by $\varepsilon \sim \mathcal{N}(0,\Sigma)$ with $\Sigma^{-1} = p \Sigma^{-1}_0$, where $\Sigma^{-1}_0=(1/N_o) \operatorname{Id}_{N_o}$.
\end{itemize}
Following our theory, we attempt to recover~$\mu^\dagger$ by solving~\eqref{eq:estproblemmeasure0_bis2} using the a priori parameter choice rule~$\beta(p)= \beta_0 /\sqrt{p}$. The regularized problems are solved by the Primal-Dual-Active-Points method, \cite{neitzel_2019, pieper_walter_2021}, yielding a solution  $\bar{\mu}$. Since the action of the forward operator~$K$ on sparse measures can be computed analytically, the algorithm is implemented on a grid free level. In addition, we compute a stationary point $\widehat{\bld{m}}$ of the nonconvex problem~\eqref{eq:estproblempoints_fixedN} inducing the measure $\widehat{\mu}$ from~\eqref{def:measparamm}. This is done by a similar iteration to the Gauss-Newton sequence \eqref{eq:gauss-newton} with a nonsmooth adaptation to handle the \(\ell_1\)-norm and an added globalization procedure to make it converge without restrictions on the data. We note that this solution depends on the initialization of the algorithm at \(\bld{m}^\dagger\), which is usually unavailable in practice.
To evaluate the reconstruction results in a qualitative way, we follow \cite{duval_peyre_2014} by considering the dual certificates and pre-certificates; see Section \ref{sec:sparseinverse}. Our Matlab implementation is available at \url{https://github.com/hphuoctruong/OED_SparseInverseProblems}.

\subsection*{Example 1. }

In the first example, we illustrate the reconstruction capabilities of the proposed ansatz  for different measurement setups and with and without noise in the observations. To this end, we attempt to recover the reference measure $\mu^{\dagger}$ using a variable number~$N_o$ of uniformly distributed sensors. For noisy data, the regularization parameter is selected as $\beta = \beta_0 / \sqrt{\ptot}$ where $\beta_0 = 2$ and $\ptot = 10^4$.
We first consider the exact measurement data with $N_o \in \{6,9,11\}$ and try to obtain~$\mu^\dagger$ by solving~\eqref{def:minprob}. The results are shown in Figure~\ref{fig:example_1A}. We observe that with $6$ sensors, the pre-certificate~$\eta_{\PC}$ is not admissible. Recalling~\cite[Proposition 7]{duval_peyre_2014}, this implies that~$\mu^\dagger$ is not a minimum norm solution. In contrast, the experiments with $9$ and $11$ uniform sensors provide admissible pre-certificates. In these situations, the pre-certificates coincide with the minimum norm dual certificates and the ground truth~$\mu^\dagger$ is indeed an identifiable minimum norm solution.
\begin{figure}[ht]
  \centering
\includegraphics[width=0.95\textwidth]{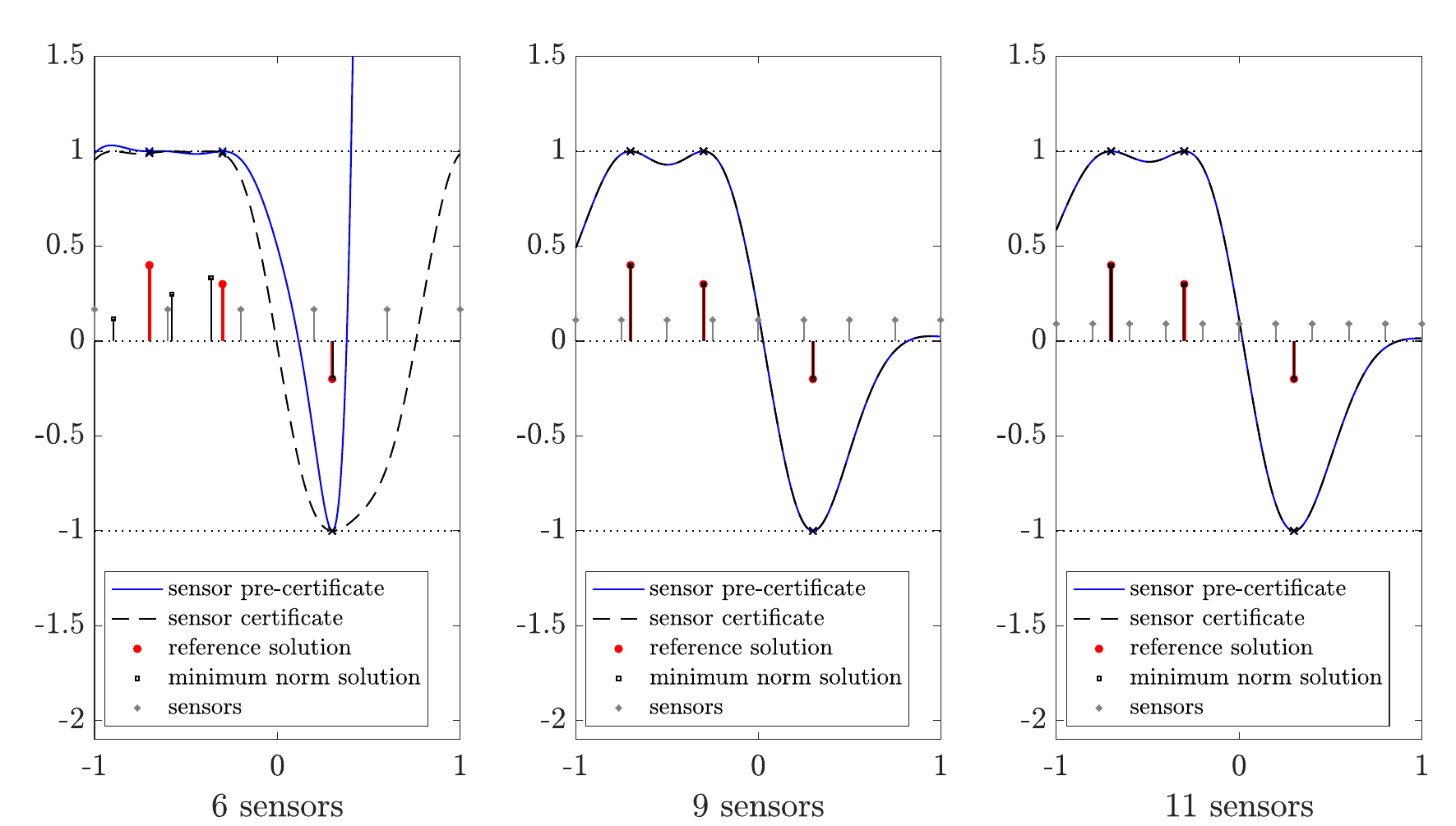}
  \caption{Reconstruction results with exact data using 6 sensors (left), 9 sensors (middle) and 11 sensors (right)
  }
  \label{fig:example_1A}
\end{figure}

\begin{figure}[ht!]
  \centering
\includegraphics[width=0.95\textwidth]{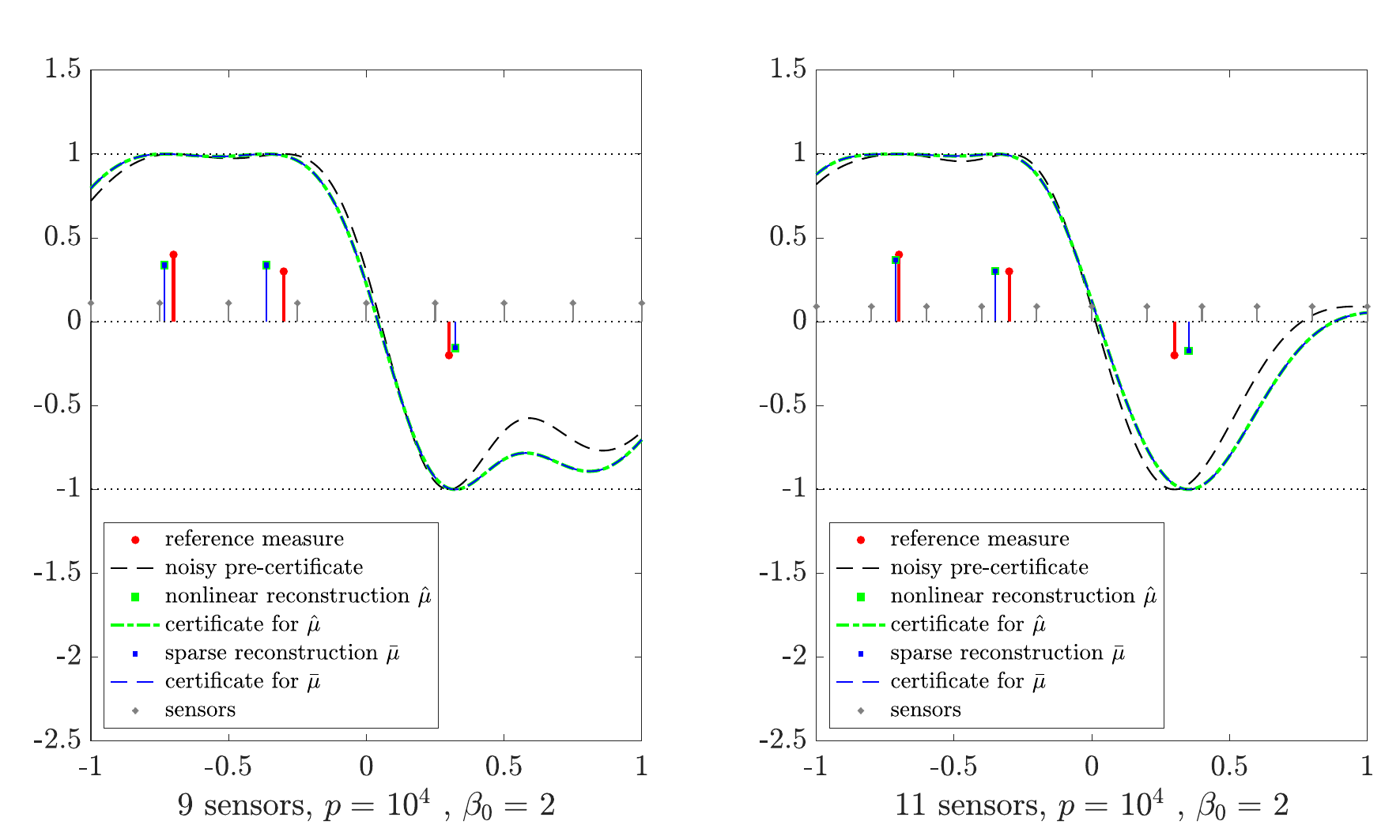}
  \caption{Reconstruction results with noisy data using $9$ sensors (left) and $11$ sensors (right)}
  \label{fig:example_1B}
\end{figure}

Next, we consider noisy data and solve~\eqref{eq:estproblemmeasure0_bis2} for the aforementioned choice of~$\beta(p)$. Following the observation in the first example, we only evaluate the reconstruction results obtained by $9$ and $11$ uniform sensors. In the absence of the measurement data obtained from experiments, we generate synthetic noisy measurements where the noise vector $\epsilon$ is a realization of the Gaussian random noise $\varepsilon \sim \mathcal{N}(0,\Sigma)$. The results are shown in Figure~\ref{fig:example_1B}. Since~$\mu^\dagger$ is identifiable in these cases, $\widehat{\mu}$ and~$\bar{\mu}$ coincide and closely approximate~$\mu^\dagger$ with high probability for an appropriate choice of~$\beta_0$ and~$p$ large enough. Both properties can be clearly observed in the plots, where \(\beta_0 = 2\).

\subsection*{Example 2.} In the second example we study the influence of the parameter choice rule on the reconstruction result. To this end, we fix the measurement setup to $9$ uniformly distributed sensors. We recall that the a priori parameter choice rule is given by $\beta(p) = \beta_0/\sqrt{p}$. According to Section \ref{sec:quantativeerror}, selecting a sufficiently large value for $\beta_0$ is recommended to achieve a high quality reconstruction. To determine a useful range of regularization parameters, we solve problem \eqref{eq:estproblemmeasure0_bis2} for a sequence of regularization parameters using PDAP. Here, we choose $\beta_0 \in \{ 0.5, 1, 2\}$ and $p \in \{ 10^4, 10^5, 10^6\}$.

\begin{figure}[ht!]
\centering
\includegraphics[width=0.95\textwidth]{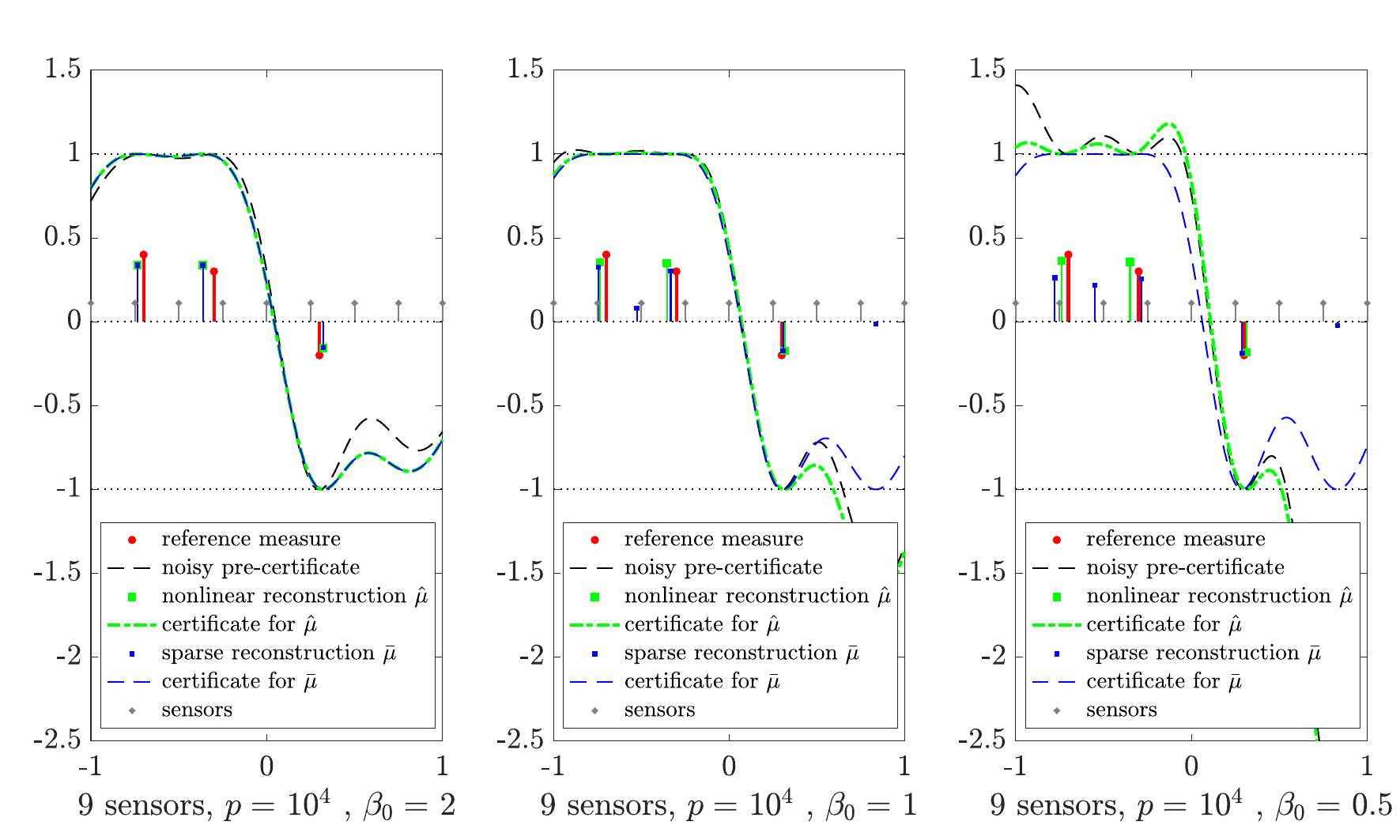}
  \caption{Reconstruction results with $\beta_0 = 2$ (left), $\beta_0 = 1$ (middle) and $\beta_0 = 0.5$ (right)}
  \label{fig:example_2}
\end{figure}
In Figure~\ref{fig:example_2}, different reconstruction results are shown for the same realization of noise,~$\beta_0 \in \{0.5, 1, 2\}$ and~$p=10^{4}$. As one can see, for this particular realization of the noise, the number of spikes is recovered exactly in the case $\beta_0 = 2$ and we again observe that $\widehat{\mu}=\bar{\mu}$.
In contrast, for smaller $\beta_0$, the noisy pre-certificate is not admissible. Hence, while $\widehat{\mu}$ still provides a good approximation of~$\mu^\dagger$,~$\bar{\mu}$ admits two additional spikes away from the support of~$\mu^\dagger$. These observations can be explained by looking at Theorem \ref{thm:MSE_bound} the second term on the right hand side of the inequality becomes negligible for increasing~$\beta_0$ and large enough~$p$. Thus, roughly speaking, the parameter $\beta_0$ controls the probability of the ``good events'' in which~$\widehat{\mu}$ is the unique solution of~\eqref{eq:estproblemmeasure0_bis2}.

Finally, we address the reconstruction error from a quantitative perspective. For this purpose, we simplify the evaluation of the maximum mean-squared error (MSE) by inserting the solution \(\bar{\mu}\) computed algorithmically. We note that this could only lead to an under-estimation of the maximum error in the case of non-unique solutions of~\eqref{eq:estproblemmeasure0_bis2}; a degenerate case that is unlikely to occur in practice. Moreover, the expectation is approximated using $10^3$ Monte-Carlo samples. Additionally, we use the closed form expression~\eqref{eq:closedformlin} for evaluating the linearized estimate~$\mathbb{E}_{\gamma_p}[\norm{\delta \widehat{\bld{m}}}_{W_{\dagger}}^2]$ exactly. Here, the expectations are computed for $\beta_0 \in\{ 2, 0.5\}$.
The results are collected in Table~\ref{tab:example1_noisyreconstruction}.
\begin{table}[ht]
\begin{tabular}{lccccc}
\toprule
& \multicolumn{1}{c}{} 
& $p = 10^4$            & $p = 10^5$              & $p = 10^6$              & $p = 10^7$              \\ 
\cmidrule{3-6}
& $\mathbb{E}_{\gamma_p}[\norm{\delta \widehat{\bld{m}}}^2_{W_{\dagger}}]$                   
& \num{7.972584e-03}    & \num{7.972584e-04}    & \num{7.972584e-05}    & \num{7.972584e-06}  \\ 
& $\mathbb{E}_{\gamma_p}[d_{\HK}(\mu^{\dagger},\widehat{\mu})^2]$                                  
& \num{5.428403e-03}    & \num{6.487211e-04}    & \num{7.347171e-05}    & \num{7.760411e-06}  \\  
\multirow{-3}{*}{$\beta_0 = 2$}                         
& $\mathbb{E}_{\gamma_p}[d_{\HK}(\mu^{\dagger},\bar{\mu})^2]$                                  
& \num{5.435915e-03}    & \textbf{\num{6.508708e-04}}    & \textbf{\num{7.416505e-05}}    & \textbf{\num{7.987016e-06}}    \\ 
\midrule
\multicolumn{1}{l}{}
& $\mathbb{E}_{\gamma_p}[\norm{\delta \widehat{\bld{m}}}^2_{W_{\dagger}}]$                          
& \num{2.074440e-03}    & \num{2.074440e-04}    & \num{2.074440e-05}    & \num{2.074440e-06}    \\
\multicolumn{1}{l}{}            
& $\mathbb{E}_{\gamma_p}[d_{\HK}(\mu^{\dagger},\widehat{\mu})^2]$                                  
& \num{1.713474e-03}    & \num{1.940421e-04}    & \num{2.030982e-05}    & \num{2.063416e-06}    \\ 
\multicolumn{1}{l}{\multirow{-3}{*}{$\beta_0 = 0.5$}} & $\mathbb{E}_{\gamma_p}[d_{\HK}(\mu^{\dagger},\bar{\mu})^2]$
& \textbf{\num{4.123229e-03}}      & \num{9.828361e-04}    & \num{2.647033e-04}    & \num{7.936999e-05}   \\
\bottomrule
\end{tabular}
\vspace{.2em}
\caption{Reconstruction results with $\beta_0 = 2$ and $\beta_0 = 0.5$.}
\label{tab:example1_noisyreconstruction}
\end{table}
We make several observations: Clearly, the MSE decreases for increasing~$p$, i.e.\ lower noise level.
For increased \(\beta_0\), the behavior differs: For the theoretical quantities \(\widehat{\bld{m}}\) and \(\delta\widehat{\bld{m}}\) increased \(\beta_0\) only introduces additional bias and thus increases error. For the estimator \(\bar{\mu}\), the increased regularization however leads to generally improved results, since the probability of \(\widehat{\mu}\neq \bar{\mu}\) is decreased. We highlight in bold the estimator which performed best for each \(\beta_0\).
Here, the results conform to Theorem~\ref{thm:MSE_bound}: For larger~$\beta_0$, the second term on the right-hand side of~\eqref{eq:estimate_maintheorem} is negligible and the linearized estimate provides an excellent bound on the MSE for both \(\widehat\mu\) and \(\bar\mu\). We also note that the estimate is closer to the MSE in the limiting case for larger \(p\). In contrast, for~$\beta=0.5$, the linearized estimate and the MSE of \(\widehat\mu\) are much smaller than the MSE of the estimator \(\bar{\mu}\). This underlines the observation that Theorem~\ref{thm:uniqueandest} requires further restrictions on the admissible noises in comparison to Proposition~\ref{prop:pertub}.

\subsection*{Example 3.} The final example is devoted to compare the reconstruction results obtained by uniform designs and an improved design chosen by heuristics. To this end, we consider three measurement setups: uniformly distributed setups with $6$ and $11$ sensors, respectively, and one with $6$ sensors selected on purpose. More precisely, in the later case, we place the sensors at $\Omega_o = \{-0.8, -0.6, -0.4, -0.1, 0.1, 0.4 \}$. The different error measures are computed as in the previous example and the results are gathered in Table \ref{tab:compareresults}.
\begin{table}[ht!]
\begin{tabular}{ccccc}
\toprule
\multicolumn{2}{l}{}                                                                                         & 11 sensors         & ``selected'' 6 sensors                    & 6 sensors              \\
\cmidrule{3-5}
\multicolumn{1}{c}{}                                                                            & $p = 10^4$ & \num{5.033998e-03} & {\bf \num{4.248795e-03}} & \num{1.541210e-02}     \\ 
\multicolumn{1}{c}{}                                                                            & $p = 10^5$ & \num{5.613522e-04} & {\bf \num{4.581969e-04}} & \num{1.193384e-02}     \\ 
\multicolumn{1}{c}{\multirow{-3}{*}{$\mathbb{E}_{\gamma_p}[d_{\HK}(\mu^{\dagger},\bar{\mu})^2]$}}          & $p = 10^6$ & \num{6.305026e-05} & {\bf \num{4.654793e-05}} & \num{1.177401e-02}     \\ 
\midrule
\multicolumn{1}{c}{}                                                                            & $p = 10^4$ & \num{6.089430e-03} & {\num{4.768186e-03}} &                        \\ 
\multicolumn{1}{c}{}                                                                            & $p = 10^5$ & \num{6.089430e-04} & {\num{4.768186e-04}} &                        \\ 
\multicolumn{1}{c}{\multirow{-3}{*}{$\mathbb{E}_{\gamma_p}[\norm{\delta \widehat{\bld{m}}}^2_{W_{\dagger}}]$}} & $p = 10^6$ & \num{6.089430e-05} & {\num{4.768186e-05}} & \multirow{-3}{*}{$\text{Inf}$} \\
\bottomrule
\end{tabular}
\vspace{.2em}
\caption{Reconstruction results with different sensor setups.}
\label{tab:compareresults}
\end{table}
We observe that the measurement setup with $6$ selected sensors performs better than the uniform ones. Moreover, the linearized estimate again provides a sharp upper bound on the error for both ten uniform and six selected sensors but yields numerically singular Fisher information matrices for six uniform sensors (denoted as \(\text{Inf}\) in the table), i.e.~$\mu^\dagger$ is not stably identifiable in this case. Note that the estimator \(\bar{\mu}\) still yields somewhat useful results, which are however affected by a constant error due to the difference in minimum norm solution and exact source as depicted in Figure~\ref{fig:example_1A} and do not improve with lower noise level. These results suggest that the reconstruction quality does not only rely on the amount of measurements taken but also on their specific setup. In this case, we point out that the selected sensors are chosen to be adapted to the sources as every two sensors are placed on the two sides of every source. Thus the obtained results imply that if we have some reasonable prior information on the source positions and amplitudes, one may obtain a better sensor placement setup by incorporating it in the design of the measurement setup. This leads to the concept of optimal sensor placement problems for sparse inversion which we will consider in a future work.
\section{Conclusion}
In the present work, we have considered the inverse problem of estimating an unknown sparse signal~$\mu^\dagger$ from finitely many measurements perturbed by Gaussian random noise which was formulated as a linear, ill-posed operator equation in the space of Radon measures. The main result of the paper is an asymptotical sharp upper bound on the mean-squared error defined in terms of the Hellinger-Kantorovich distance of a nonsmooth Tikhonov-type estimator which is confirmed by extensive numerical experiments. Its proof relies on three key concepts: A suitable a priori regularization parameter choice rule~$\beta=\beta(p)$ which is adapted to the overall precision of the measurements~$p$, the non-degeneracy of the minimal-norm dual certificate as well as a careful linearization argument for the H-K distance on a quantifiable set of random events. In comparison to the intractable mean-squared error, the new bound is easily computable and explicitly depends on the locations of the measurement sensors as well as their relative precision. In perspective, these observations suggest the application of this new-found upper estimate in the context of optimal sensor design for sparse inverse problems. However, we also point out that a practical realization of such an approach is not straightforward since the derived upper bound, i.e.\ the prospective design criterion, depends on the unknown source~$\mu^\dagger$ and the non-degeneracy of the minimal-norm certificate, a property that also inherently depends on the measurement setup. Addressing these problems goes beyond the scope of the current paper and will be addressed in future work. Moreover, the extension of the presented result towards vector measures as, e.g., encountered in acoustic inversion is of great interest.

\vspace{1cm}
\noindent \textbf{Acknowledgments.} The work of P.-T. Huynh was supported by the Austrian Science Fund FWF under the grant DOC 78. The material in this manuscript is based on work supported by the Laboratory Directed Research and Development Program at Oak Ridge National Laboratory (ORNL), managed by UT-Battelle, LLC, under Contract No.\ DE--AC05--00OR22725. The US government retains and the publisher, by accepting the article for publication, acknowledges that the US government retains a nonexclusive, paid-up, irrevocable, worldwide license to publish or reproduce the published form of this manuscript, or allow others to do so, for US government purposes. DOE will provide public access to these results of federally sponsored research in accordance with the DOE Public Access Plan (\url{http://energy.gov/downloads/doe-public-access-plan}).


\bibliography{mybibfile}
\bibliographystyle{abbrv}
\newpage
\appendix

\section{Auxiliary results}



\subsection{Gaussian tail bounds}
The following lemma gives an estimate on the tail probabilities of a Gaussian random variables as well as on its moments. 
\begin{lemma}\label{lem:probabilitybound} Assume that $\varepsilon_0 \sim \gamma_E=\mathcal{N}(0,  \operatorname{Id}_{N_o})$. Then for every $\alpha > 0$ there holds
\[
\mathbb{P}(\norm{\varepsilon_0}_2 > \alpha) \le 2 \exp \left( -\frac{\alpha^2}{2N_o}\right).
\]
Moreover for \(l\geq 1\), with \(C_l = (2l-1)!! = (2l-1)(2l-3)\cdots 1\), we have
\[
\int_{\norm{\epsilon_0}_2 > \alpha} \norm{\epsilon_0}^l_2 \de \gamma_E(\epsilon_0)
\leq  \sqrt{2 N_o C_l} \exp \left( -\frac{\alpha^2}{4N_o}\right)
= \exp \left( -\frac{\alpha^2}{4N_o} + \log(2 N_o C_l)/2\right)
.
\]
\end{lemma}
\begin{proof}
According to Remark~4 in~\cite{pinelis}, we get for any~\(\lambda>0\) that
\begin{align*}
\mathbb{P}(\norm{\varepsilon_0}_2 > \alpha)
=\int_{\norm{\epsilon_0} > \alpha} \de \gamma_E(\epsilon_0)
\leq 2 \exp\left(-\lambda \alpha\right) \exp\left(\lambda^2 N_o/2 \right).
\end{align*}
Minimizing the right-hand side with respect to ~$\lambda$ yields $\lambda = \alpha/N_o$ and the first estimate.
The second inequality is due to
\[
\int_{\norm{\epsilon_0} > \alpha} \norm{\epsilon_0}^l_2 \de \gamma_E(\epsilon_0)
\leq \sqrt{\int \norm{\epsilon_0}^{2l}_2 \de \gamma_E(\epsilon_0)} \sqrt{2 \exp \left( -\frac{\alpha^2}{2N_o}\right)}
\]
with Cauchy-Schwarz. The proof is finished noting that \(\E[ \norm{\varepsilon_0}^{2l}_2 ] = N_o C_l\) where \(C_l = (2l-1)!!\) denotes the \(2l\)-the moment of the univariate standard normal distribution.
\end{proof}

\subsection{Some results on measurability} \label{app:meas}
In this section we address the measurability of the worst-case distance function
\begin{align} \label{def:worstcasefunc}
\epsilon \mapsto \sup_{\miu \in\mathfrak{M}(\epsilon)} d_{\HK}(\miu ,\miu^\dagger)  
\end{align}
as well as the boundedness of its second moment.
For this purpose, recall the definition of the solution set
\begin{align*}
\mathfrak{M}(\epsilon)= \argmin_{\miu\in\mathcal{M}(\Omega_s)} \left \lbrack \dfrac{1}{2} \norm*{\Ke \miu - \zd(\varepsilon)}_{\Prec}^2 + 
\beta\norm{\miu}_{\M(\Omega_s)} \right\rbrack
\end{align*}
and note that~$\mathfrak{M}(\epsilon)$ is weak* compact. We first show that the supremum in~\eqref{def:worstcasefunc} is attained and give a useful upper bound on it.
\begin{lemma}
\label{lem:supremizing_solution}
    For every~$\epsilon \in \mathbb{R}^{N_o}$, there is~$\bar{\mu}(\epsilon)\in \mathfrak{M}(\epsilon)$ with
    \begin{align*}
        d_{\HK}(\bar{\mu}(\epsilon),\miu^\dagger) =\sup_{\miu \in \mathfrak{M}(\epsilon)} d_{\HK}(\miu,\miu^\dagger).
    \end{align*}
    Moreover, we have
        \begin{align} \label{eq:upperonworst}
\sup_{\miu \in \mathfrak{M}(\epsilon)} d_{\HK}(\miu,\miu^\dagger)^2
\leq 2\norm{\miu^\dagger}_{\mathcal{M}(\Omega_s)} + \frac{1}{2\beta} \norm{\epsilon}_{\Sigma_0^{-1}}^2.
\end{align}
\end{lemma}
\begin{proof}
    The inequality~\eqref{eq:upperonworst} follows from Proposition~\ref{prop:bound_on_barq} and \cite[Proposition~7.18]{LieroMielkeSavare:2018}. Hence, there exists a supremizing sequence $\{\miu_k\}_k \subset \mathfrak{M}(\epsilon)$, i.e,
    \begin{align*}
        \lim_{k\rightarrow \infty} d_{\HK}(\miu_k,\miu^\dagger)=\sup_{\miu \in \mathfrak{M}(\epsilon)} d_{\HK}(\miu,\miu^\dagger).
    \end{align*}
    Due to weak* compactness of~$\mathfrak{M}(\epsilon)$, it admits a subsequence, denoted by the same subscript, as well as~$\bar{\mu}(\epsilon) \in \mathfrak{M}(\epsilon)$ with~$\miu_k \rightharpoonup^* \bar{\mu}(\epsilon)$. Since~$d_{\HK}$ metrizes weak* convergence, the inverse triangle inequality yields
    \begin{align*}
        |d_{\HK}(\miu_k,\miu^\dagger)-d_{\HK}(\bar{\mu}(\epsilon),\miu^\dagger)| \leq d_{\HK}(\bar{\mu}(\epsilon),\miu_k) \rightarrow 0
    \end{align*}
    and thus
    \begin{align*}
       d_{\HK}(\bar{\mu}(\epsilon),\miu^\dagger)=\lim_{k\rightarrow \infty} d_{\HK}(\miu_k,\miu^\dagger)=\sup_{\miu \in \mathfrak{M}(\epsilon)} d_{\HK}(\miu,\miu^\dagger).
       \qquad\qedhere
    \end{align*}
\end{proof}
Using Lemma~\ref{lem:supremizing_solution}, we conclude the measurability of the worst-case distance.
\begin{proposition}
    The function defined in~\eqref{def:worstcasefunc} is ~$\gamma_p$-measurable. Moreover, there holds
    \begin{align*}
       \mathbb{E}_p \left \lbrack \sup_{ \mu \in \mathfrak{M}(\cdot)} d_{\HK}(\mu,\mu^\dagger )^2\right \rbrack \leq 2\norm{\miu^\dagger}_{\mathcal{M}(\Omega_s)} + \frac{N_0}{2\beta \ptot} < \infty
    \end{align*}
\end{proposition}
\begin{proof}
For abbreviation, define
\begin{align*}
    \operatorname{Err}[\bar{\mu}](\epsilon)= \sup_{\miu \in \mathfrak{M}(\epsilon)} d_{\HK}(\miu,\miu^\dagger)
\end{align*}
and let~$\{\epsilon_k\}_k$ denote a convergent sequence with limit~$\epsilon$. Note that the set~$\bigcup_k \mathfrak{M}(\epsilon_k)$ as well as the sequence~$\operatorname{Err}[\bar{\mu}](\epsilon_k)$ are bounded. Now, choose a subsequence~$\{\epsilon_{k,i}\}_i$ such that
    \begin{align*}
        \lim_{i \rightarrow \infty} \operatorname{Err}[\bar{\mu}](\epsilon_{k,i})=\limsup_{k \rightarrow \infty} \operatorname{Err}[\bar{\mu}](\epsilon_{k})
    \end{align*}
    and let~$\{\bar{\mu}(\epsilon_{k,i})\}_i$ be a corresponding sequence of maximizers from Lemma~\ref{lem:supremizing_solution}. By possibly extracting another subsequence, there is~$\Tilde{\mu}$ with~$\bar{\mu}(\epsilon_{k,i}) \rightharpoonup^* \Tilde{\miu}$. By standard arguments, we verify that~$\Tilde{\miu} \in \mathfrak{M}(\epsilon)$. Consequently, we have
    \begin{align*}
        \limsup_{k \rightarrow \infty} \operatorname{Err}[\bar{\mu}](\epsilon_{k})=\lim_{i \rightarrow \infty} \operatorname{Err}[\bar{\mu}](\epsilon_{k,i})=\lim_{i \rightarrow \infty} d_{\HK}(\miu(\epsilon_{k,i}), \miu^\dagger)^2 =d_{\HK}(\Tilde{\miu}, \miu^\dagger)^2 \leq \operatorname{Err}[\bar{\mu}](\epsilon).
    \end{align*}
    Hence $\operatorname{Err}[\bar{\mu}]$ is upper semicontinuous and thus measurable w.r.t~$\gamma_p$. Finally, we apply~\eqref{eq:upperonworst} to conclude
    \begin{align*}
        \mathbb{E}_p \left \lbrack \sup_{ \mu \in \mathfrak{M}(\cdot)} d_{\text{HK}}(\mu,\mu^\dagger )^2\right \rbrack \leq 2\norm{\miu^\dagger}_{\mathcal{M}(\Omega_s)}
        + \frac{1}{2\beta} \mathbb{E}_p \left \lbrack \norm{\varepsilon}_{\Sigma_0^{-1}}^2 \right \rbrack
        \leq  2\norm{\miu^\dagger}_{\mathcal{M}(\Omega_s)} + \frac{N_0}{2\beta \ptot}.
        \qquad\qedhere
    \end{align*}    
\end{proof}
\section{Proofs of Proposition} \label{app:complicatedproofs}
In this section we provide the omitted proofs of Proposition~\ref{prop:pertub} and~\ref{prop:admissibility}, respectively, as well as all the auxiliary results needed in their derivation.
\begin{proposition}\label{prop:estimate_G} The following estimates hold:
\begin{align*}
\norm{\Sigma^{-1/2}_0 G'(\bld{m}) \delta\bld{m}}_{2} 
&\leq C_k \norm{\delta\bld{q}/w}_2 + C'_k \sqrt{\norm{\bld{q}}_1} \norm{w \delta\bld{y}}_{2}, \\
 \norm{\Sigma^{-1/2}_0 G''(\bld{m}) (\delta\bld{m}, \tau\bld{m})}_2
&\leq C'_k (\norm{\delta\bld{q}/w}_2\norm{w\tau\bld{y}}_{2} + \norm{\tau\bld{q}/w}_2\norm{w\delta\bld{y}}_{2}) + C''_k \norm{w\delta\bld{y}}_{2}\norm{w\tau\bld{y}}_{2},
\end{align*}
where \(w_n = \sqrt{\abs{q_n}}\).
In particular, with the \(W\)-norm \(\norm{\delta \bld{m}}^2_W := \norm{\delta\bld{q}/w}_2^2/4 + \norm{w \delta\bld{y}}_2^2\) with \(W = W(\bld{m})\), we have
\begin{align}
\label{eq:estimateG_21} \norm{\Sigma_0^{-1/2}G'(\bld{m})}_{W \to 2} &\leq (2C_k  + C'_k)\sqrt{\norm{\bld{q}}_1},\\
\label{eq:estimateG_22} \norm{\Sigma_0^{-1/2}G''(\bld{m})}_{W\times W \to 2} &\le 4 C'_k + C''_k.
\end{align}
\end{proposition}

\begin{proof}
We first notice that \(\norm{v}_{\Sigma_0^{-1}} \leq \norm{v}_\infty\) due to \(\tr\Prec = 1\). In addition, one can write
\begin{align*}
[G'(\bld{m})\delta \bld{m}]_k &= \sum_{n=1}^{N_s} \ke(x_k,y_n) \delta q_n + (\nabla_y \ke(x_k,y_n))^\top \delta y_n \, q_n  \\
&= \sum_{n=1}^{N_s} \ke(x_k,y_n) w_n \, \delta q_n / w_n + (\nabla_y \ke(x_k,y_n))^\top w_n \delta y_n \, q_n / w_n  \\
[G''(\bld{m})(\delta \bld{m}, \tau \bld{m})]_k &= \sum_{n=1}^{N_s} (\nabla_y \ke(x_k,y_n))^\top \delta y_n \, \tau q_n  + (\nabla_y \ke(x_k,y_n))^\top \tau y_n \, \delta q_n  \\
&+ \delta y_n^\top \nabla_{yy}^2 \ke(x_k,y_n) \tau y_n q_n.
\end{align*}
Here, we choose \(w_n = \sqrt{\abs{q_n}}\).
Hence, by estimating term by term, we have 
\begin{align*}
\norm{G'(\bld{m}) \delta \bld{m}}_{\Sigma_0^{-1}} 
\leq C_k \norm{w}_2 \norm{\delta\bld{q}/w}_2 + C'_k \norm{\bld{q}/w}_2 \norm{w \delta\bld{y}}_2 \leq \left(2C_k + C'_k\right)\sqrt{\norm{\bld{q}}_1} \norm{\delta \bld{m}}_W,
\end{align*}
which implies \eqref{eq:estimateG_21}. A similar argument gives \eqref{eq:estimateG_22}.
\end{proof}

\begin{proposition}\label{prop:pertub_estimate}
Define the constant
\[
r^{\dagger} = r(\mu^{\dagger})
:= \min\left\{\,\min\{w_n^{\dagger}/8,\, d_{w_n^{\dagger}}(y_n^{\dagger},\partial\Omega)/2\}\;|\; n = 1,\ldots, N_s\,\right\}.
\]
Then for every $\bld{m} \in B_{W_\dagger}(\bld{m}^{\dagger}, r^{\dagger})$, there holds $\sign \bld{q} = \sign \bld{q}^{\dagger}$ and
\begin{equation}\label{eq:estimate_R}
R(\bld{q},\bld{q}^{\dagger}) \le 2 \quad\text{and }1/2 \norm{\bld{q}^{\dagger}}_1 \le \norm{\bld{q}}_1 \le 2\norm{\bld{q}^{\dagger}}_1,
\end{equation}
where \(R(\bld{q},\bld{q}^\dagger)\) is the maximal ratio of the weights \(w_n\) and \(w_n^\dagger\) from Proposition~\ref{prop:HK_distance_estimate}.

In addition, for all $\bld{m}, \bld{m}' \in B_{W_\dagger}(\bld{m}^{\dagger}, r^{\dagger})$ and \(\delta \bld{m}\), there holds 
\begin{align}
\label{eq:estimateG_31} \norm*{\Sigma_0^{-1/2} (G(\bld{m}) - G(\bld{m}'))}_2 &\leq L_G \norm{\bld{m} - \bld{m}'}_{W_\dagger},\\
\label{eq:estimateG_32}
\norm*{\Sigma_0^{-1/2} G'(\bld{m})\delta\bld{m} }_2 & \leq L_G \norm{\delta \bld{m}}_{W_\dagger},\\
\label{eq:estimateG_33}
 \norm*{\Sigma_0^{-1/2} (G'(\bld{m}) - G'(\bld{m}')) \delta\bld{m} }_2 &\leq L_{G'} \norm{\bld{m} - \bld{m}'}_{W_\dagger} \norm{\delta\bld{m}}_{W_\dagger},
\end{align}
where $L_G := 4(2C_k  + C'_k)\sqrt{\norm{\bld{q^{\dagger}}}_1}$ and $L_{G'} := 2(4 C'_k + C''_k)$.
\end{proposition}

\begin{proof} For $\bld{m} \in B_{W_\dagger}(\bld{m}^{\dagger}, r)$, one has
\[
\abs{q_n - q^{\dagger}_n}/w^{\dagger}_n
\leq \norm{(\bld{q} - \bld{q}^{\dagger})/w^{\dagger}}_2
\le 4\norm{\bld{m} - \bld{m}^{\dagger}}_{W_\dagger} \le 1/2 w_n^{\dagger},\quad \forall n = 1,\ldots, N_s.
\]
This implies $1/2 \le q_n/q_n^{\dagger} \le 3/2$ for all $ n = 1, 2,\ldots, N_s$. Hence, $\sign \bld{q} = \sign \bld{q}^{\dagger}$ and~\eqref{eq:estimate_R} follows. Also, the condition $r^{\dagger} \le d_{w_n^{\dagger}}(y_n^{\dagger},\partial\Omega)/2$ guarantees that $y_n \in \Omega_s$ for all $n = 1, \ldots, N_s$.
By Proposition~\ref{prop:estimate_G} and \eqref{eq:estimate_R}, it can now be seen that
\begin{equation}\label{eq:est_GmGm'}
\begin{aligned}
\norm*{\Sigma_0^{-1/2} (G(\bld{m}) - G(\bld{m}'))}_2 
&= \norm*{\Sigma_0^{-1/2} \int_0^1 G'\left(\bld{m}' + t(\bld{m} - \bld{m}') \right)\de t (\bld{m} - \bld{m}') }_2 \\
& \le \int_0^1 \norm*{\Sigma_0^{-1/2} G'\left(\bld{m}' + t(\bld{m} - \bld{m}') \right)}_{W_\dagger \to 2} \de t \norm{\bld{m} - \bld{m}' }_{W_\dagger},
\end{aligned}   
\end{equation}
for every \(\bld{m},\bld{m}' \in B_{W_{\dagger}}(\bld{m}^{\dagger},r^{\dagger})\). Next, since $\bld{m}' + t(\bld{m} - \bld{m}') \in B_{W_{\dagger}}(\bld{m}^{\dagger},r^{\dagger})$, for \(W = W(\bld{m}' + t(\bld{m} - \bld{m}')) \) and $W_{\dagger} = W_{\dagger}(\bld{m}^{\dagger})$, we use \eqref{eq:equi_norm}, \eqref{eq:estimate_R} and \eqref{eq:estimateG_21} to deduce that
\begin{equation}\label{eq:est_Gm2}
\begin{aligned}
\norm*{\Sigma_0^{-1/2} G'\left(\bld{m}' + t(\bld{m} - \bld{m}') \right)}_{W_\dagger \to 2} 
&\le 2\norm*{\Sigma_0^{-1/2} G'\left(\bld{m}' + t(\bld{m} - \bld{m}') \right)}_{W \to 2} \\
&\le 2\left(2C_k + C_k'\right)\sqrt{\norm{\bld{q}' + t (\bld{q} - \bld q')}_1} \\
&\le 4\left(2C_k + C_k'\right)\sqrt{\norm{\bld{q}^{\dagger}}_1}.
\end{aligned}
\end{equation}
Combining \eqref{eq:est_GmGm'} and \eqref{eq:est_Gm2}, we deduce 
now~\eqref{eq:estimateG_31}. Similarly, \eqref{eq:estimateG_32} follows from \eqref{eq:estimateG_21} and \eqref{eq:estimate_R} as well. Moreover, \eqref{eq:estimateG_33} can be proved using the estimate~\eqref{eq:estimateG_22} with a similar argument.
\end{proof}
\begin{proof}[\sc Proof of Proposition~\ref{prop:pertub}] Since $G'(\bld{m}^{\dagger})$ has full column rank, the Fisher information matrix $\mathcal{I}_0$ defined in \eqref{eq:fisher} is invertible. Hence, the map \(T(\bld{m}) := \bld{m} - \mathcal{I}_0^{-1}S(\bld{m})\) is well-defined, where $S(\bld{m})$ is the residual of the stationarity equation given in \eqref{eq:stationarity_R} with \(\bar{\bld{\rho}} = \bld{\rho} = \sign \bld{q}^\dagger\). In order to obtain the claimed results, we aim to show that \(T \) is a contraction and argue similarly to the proof of the Banach fixed point theorem. However, since the correct domain of definition for the map \(T\) is difficult to determine beforehand, we provide a direct proof. 

We start by showing that~$T$ is Lipschitz continuous on the ball ${B_{W^{\dagger}}}(\bld{m}^{\dagger},\hat{r})$ for some as of yet undetermined \(0 < \hat{r} \leq r\) with Lipschitz constant~$\kappa(\hat{r}) \leq 1/2$ if~$\eps$ is chosen suitably. For this purpose, consider two points \(\bld{m}\) and \(\bld{m}'\) in ${B_{W^{\dagger}}}(\bld{m}^{\dagger},\hat{r})$, their difference \(\delta\bld{m} = \bld{m}-\bld{m}'\) and the difference of their images \(\delta\bld{m}_T = T(\bld{m})-T(\bld{m}')\). Note that
\begin{equation*}
\begin{aligned}
 \mathcal{I}_0\delta \bld{m}_T
 &= \mathcal{I}_0\delta \bld{m} - (S(\bld{m}) - S(\bld{m}')) \\
 &= (G'(\bld{m}^\dagger) - G'(\bld{m}))^\top \Prec G'(\bld{m}^\dagger) \delta\bld{m}\\
 &\quad + G'(\bld{m})^\top \Prec (G'(\bld{m}^\dagger)\delta\bld{m} - (G(\bld{m}) - G(\bld{m}'))) \\
 &\quad - (G'(\bld{m}) - G'(\bld{m}'))^\top \Prec (G(\bld{m}') - G(\bld{m}^{\dagger}) - \epsilon).
\end{aligned}
\end{equation*}
We multiply this equation from the left with \((\delta \bld{m}_T)^\top\) and consider each term on the right hand side separately. Using Proposition \ref{prop:pertub_estimate}, we have for the first term
\begin{equation*}
\begin{aligned}
&(\delta\bld{m}_T)^\top (G'(\bld{m}^\dagger) - G'(\bld{m}))^\top \Sigma_0^{-1} G'(\bld{m}^\dagger) \delta\bld{m} \\
&= \left(\Sigma_0^{-1/2} (G'(\bld{m}^\dagger) - G'(\bld{m}))\delta\bld{m}_T\right)^\top  \Sigma_0^{-1/2} G'(\bld{m}^\dagger) \delta\bld{m} \\
&\leq \norm{\Sigma_0^{-1/2} (G'(\bld{m}^\dagger) - G'(\bld{m}))\delta\bld{m}_T}_2 \norm{\Sigma_0^{-1/2} G'(\bld{m}^\dagger) \delta\bld{m}}_2 \\
&\leq L_G L_{G'} \norm{\bld{m}^\dagger - \bld{m}}_{W_\dagger}\norm{\delta\bld{m}}_{W_\dagger}\norm{\delta\bld{m}_T}_{W_\dagger}.
\end{aligned}
\end{equation*}
For the second term we estimate
\begin{equation*}
\begin{aligned}
&(\delta\bld{m}_T)^\top G'(\bld{m})^\top \Sigma_0^{-1} (G'(\bld{m}^\dagger)\delta\bld{m} - (G(\bld{m})- G(\bld{m}')))\\
&= \left(\Sigma_0^{-1/2}G'(\bld{m})\delta\bld{m}_T\right)^\top\Sigma_0^{-1/2} \int_0^1 (G'(\bld{m}^\dagger) - G'(\tau\bld{m} + (1-\tau)\bld{m}^\prime)) \delta\bld{m} \de \tau \\
&\leq L_{G} L_{G'} \int_0^1\norm{\bld{m}^\dagger - (\tau\bld{m} + (1-\tau)\bld{m}^\prime)}_{W_\dagger} \de \tau  \norm{\delta\bld{m}}_{W_\dagger}\norm{\delta\bld{m}_T}_{W_\dagger}
\end{aligned}
\end{equation*}
and for the third term we have
\begin{equation*}
\begin{aligned}
 &(\delta\bld{m}_T)^\top(G'(\bld{m}) - G'(\bld{m}'))^\top \Sigma_0^{-1} (G(\bld{m}') - G(\bld{m}^{\dagger}) - \epsilon)\\
 &= \left(\Sigma_0^{-1/2}(G'(\bld{m}) - G'(\bld{m}'))\delta\bld{m}_T\right)^\top \Sigma_0^{-1/2}(G(\bld{m}') - G(\bld{m}^{\dagger}) - \epsilon)
  \\
 &\leq
 L_{G'} (L_{G}\norm{\bld{m}^\dagger - \bld{m}'}_{W_\dagger} + \norm{\epsilon}_{\Sigma^{-1}_0})  \norm{\delta\bld{m}}_{W_\dagger} \norm{\delta\bld{m}_T}_{W_\dagger}.
\end{aligned}
\end{equation*}
Since \(\bld{m},\bld{m}'\) are contained in the ball ${B_{W^{\dagger}}}(\bld{m}^{\dagger},\hat{r})$ it follows that
\begin{align*}
\normI^{-1} \norm{\bld{m}_T}^2_{W_\dagger}
&\leq (\delta \bld{m}_T)^{\top} \mathcal{I}_0 \delta \bld{m}_T
\leq L_{G'}\left(3L_G \hat{r} + \norm{\epsilon}_{\Sigma^{-1}_0}\right)\norm{\delta \bld{m}_T}_{W_\dagger} \norm{\delta \bld{m}}_{W_\dagger}  ,
\end{align*}
using the fact that one has
\begin{align*}
\bld{m}^{\top} \mathcal{I}_0 \bld{m} &= (W_\dagger^{1/2}\bld{m})^{\top} \left[W_\dagger^{-1/2}\mathcal{I}_0 W_\dagger^{-1/2}\right] (W_\dagger^{1/2}\bld{m}) 
\\
& \ge \norm{\bld{m}}_{W_\dagger}^2 \norm{W_\dagger^{1/2} \mathcal{I}_0^{-1} W_\dagger^{1/2}}_{2 \to 2}^{-1} = \norm{\bld{m}}_{W_\dagger}^2 \normI^{-1}.
\end{align*}
Dividing by \(\norm{\bld{m}_T}_{W_\dagger}\), the estimate
\[
\norm{T(\bld{m}) - T(\bld{m}')}_{W_\dagger} = \norm{\delta\bld{m}_T}_{W_\dagger} \leq \kappa(\hat{r}) \norm{\delta\bld{m}}_{W_\dagger} = \kappa(\hat{r}) \norm{\bld{m} - \bld{m}'}_{W_\dagger}
\]
follows with
\begin{align*}
\kappa(\hat{r}) := L_{G'}\normI \left(3L_G \hat{r} + \norm{\epsilon}_{\Sigma^{-1}_0}\right).
\end{align*}

The contraction estimate above holds for any \(\hat{r} \leq r\) under the assumption that the points under consideration lie in the appropriate ball. In order to ensure contraction, we need to establish an appropriate bound and assumptions on the data. For this, we consider the linearized estimate
\begin{align*}
\delta \widehat{\bld{m}} 
= -\mathcal{I}_0^{-1}S(\bld{m}^\dagger) 
&= \mathcal{I}_0^{-1}\left[G'(\bld{m}^{\dagger})^{\top} \Prec\epsilon - \beta \vectorsgn \right],
\end{align*}
from~\eqref{eq:delta_z_hat}. Using the weighted $W_\dagger$-norm defined in Proposition~\ref{prop:estimate_G}, one has
\begin{align*}
\norm{\delta \widehat{\bld{m}}}_{W_\dagger}
&\le \normI (\norm{(\Sigma_0^{-1/2} G'(\bld{m}^{\dagger}))^{\top}\Sigma_0^{-1/2}\epsilon}_{W_\dagger^{-1}} + \beta \norm*{\vectorsgn}_{W_\dagger^{-1}} ) \\
&\le \normI (\norm{(\Sigma_0^{-1/2} G'(\bld{m}^{\dagger}))^{\top}}_{2 \to W_\dagger^{-1}} \norm{\Sigma_0^{-1/2}\epsilon}_2 + \beta  \sqrt{\norm{\bld{q}^{\dagger}}_1} )\\
&\le \normI ( L_G \norm{\epsilon}_{\Sigma_0^{-1}} + \beta \sqrt{\norm{\bld{q}^{\dagger}}_1}),
\end{align*}
where we have used \eqref{eq:estimateG_21} together with $\norm{A^{\top}}_{2 \to W_\dagger^{-1}} = \norm{A}_{W_\dagger \to 2}$ and \(\norm*{\vectorsgn}_{W_\dagger^{-1}} = \sqrt{\norm{\bld{q}^{\dagger}}_1}\).
In the following, we denote
\begin{equation*}
\begin{aligned}
\ca := \normI (L_G + \sqrt{\norm{\bld{q}^{\dagger}}_1}),\quad \cc := L_{G'} \normI (6 L_G \ca + 1).
\end{aligned}
\end{equation*}
If we now choose \( \hat{r} = \min\left\{ {\ca}/{\cc}, \, r^{\dagger}/2\right\} \) and assume that
\begin{equation}\label{eq:constant_epsilon}
    \norm{\epsilon}_{\Sigma^{-1}_0} + \beta
\leq 
\frac{\hat{r}}{2\ca}
= \min\left\{\,\frac{1}{2\cc} ,\; \frac{r^{\dagger}}{4\ca} \,\right\},
\end{equation}
then it follows immediately with the previous estimates that 
\begin{equation*}
\norm{\delta \widehat{\bld{m}}}_{W_\dagger} \le \ca ( \norm{\epsilon}_{\Sigma_0^{-1}} + \beta)
\leq \frac{\hat{r}}{2}\quad\text{and } \kappa(\hat{r}) \leq 1/2.
\end{equation*}

We are now ready to show the existence of a fixed point in~${B_{W_\dagger}}(\bld{m}^\dagger,\hat{r})$ as well as the claimed estimates.  
For this purpose, consider the simplified Gauss-Newton iterative sequence
\begin{equation}\label{eq:gauss-newton}
\bld{m}_0 = \bld{m}^{\dagger}, \quad \bld{m}^{k+1}= T(\bld{m}^k) =  \bld{m}^k - \mathcal{I}_0^{-1}S(\bld{m}^k),\quad k \ge 1.
\end{equation}
Put $\delta \bld{m}^k := \bld{m}^k - \bld{m}^{k-1}$, $k \ge 1$. It can be seen that the first Gauss-Newton step is given by \(\delta \bld{m}^1 = \delta \widehat{\bld{m}}\).
We use induction to prove that \( \bld{m}^k \in B_{W_{\dagger}}(\bld{m}^{\dagger},\hat{r})\) for all~$k\geq 0$. Indeed, if~$\eps$ satisfies \eqref{eq:constant_epsilon}, we have $\norm{\bld{m}^1 - \bld{m}^{\dagger}}_{W_\dagger} = \norm{\delta \widehat{\bld{m}}}_{W_{\dagger}} \le \hat{r}/2$, which implies $\bld{m}^1 \in {B_{W_\dagger}}(\bld{m}^{\dagger},\hat{r})$. Assume that $\bld{m}^k \in B_{W_\dagger}(\bld{m}^{\dagger},\hat{r})$. Notice that it holds
\(\norm{\delta \bld{m}^{k+1}}_{W_\dagger} = \norm{T(\bld{m}^k) - T(\bld{m}^{k-1})}_{W_\dagger} \leq \kappa \norm{\delta \bld{m}^{k}}_{W_\dagger}\).
Then, with \(d^k := \norm{\delta\bld{m}^{k}}_{W_\dagger}\) and \(e^k := \sum_{i=1}^k d^i\) we have
\begin{equation*}
d^{k+1} 
\leq \kappa d^k
\quad\text{and }
e^k \leq \frac{1-\kappa^k}{1-\kappa} d^1 \leq  \frac{1}{1-\kappa} d^1.
\end{equation*}
Hence, 
\begin{equation}\label{eq:bound_mk}
\norm{\bld{m}^{k+1} - \bld{m}^\dagger}_{W_\dagger} \le \sum_{i=1}^{k+1} \norm{\bld{m}^i - \bld{m}^{i-1}}_{W_\dagger} = e^{k} \le \dfrac{1}{1-\kappa} d^1 \le 2 \norm{\delta \widehat{\bld{m}}}_{W_\dagger} \le \hat{r},
\end{equation}
and thus $\bld{m}^{k+1} \in {B_{W_\dagger}}(\bld{m}^{\dagger},\hat{r})$.
Going to the limit, by standard arguments, we obtain that \(\bld{m}^k \to \widehat{\bld{m}} \in {B_{W_\dagger}}(\bld{m}^{\dagger},\hat{r})\) with \(T(\widehat{\bld{m}}) = \widehat{\bld{m}}\) and thus \(S(\widehat{\bld{m}}) = 0\).
Furthermore, by letting $k \to \infty$ in \eqref{eq:bound_mk}, we obtain $\norm{\widehat{\bld{m}} - \bld{m}^{\dagger}}_{W_\dagger} \le 2\norm{\delta \widehat{\bld{m}}}_{W_\dagger}$. 

For the second estimate, we rewrite the difference between the error and the perturbation in terms of all updates 
\[
\widehat{\bld{m}} - \bld{m}^\dagger - \delta\widehat{\bld{m}}
= \widehat{\bld{m}} - \bld{m}^0 - \delta \bld{m}^1 =
\sum_{k=2}^\infty \delta \bld{m}^k.
\]
Now, choosing $\bld{m}:=\bld{m}^{k+1}$ and $\bld{m}':= \bld{m}^{k}$ we have the contraction estimate
\[
\norm{\delta \bld{m}^k}_{W_\dagger} \le \kappa(\tilde{r}) \norm{\delta \bld{m}^{k-1}}_{W_\dagger},   
\]
where \(\hat{r}\) is now replaced by \(\tilde{r} = \max\{\,\norm{\bld{m}^{k+1}-\bld{m}^\dagger},\, \norm{\bld{m}^{k}-\bld{m}^\dagger}\,\} \leq 2 d^1 \leq 2 \ca (\norm{\epsilon}_{\Sigma^{-1}_0} + \beta)\) and thus \(\kappa(\tilde{r}) \leq c_2 (\norm{\epsilon}_{\Sigma^{-1}_0} + \beta)\).
Hence, bounding the updates by \(\norm{\delta \bld{m}^k}_{W_\dagger} = d^k 
\leq \kappa(\tilde{r})^{k-1} d^1 
\leq (1/2)^{k-2} \kappa(\tilde{r}) d^1 \),
we conclude
\begin{equation}\label{eq:constant_c2}
\norm{\widehat{\bld{m}} - \bld{m}^\dagger - \delta\widehat{\bld{m}}}_{W_\dagger}
\leq \sum_{k=2}^\infty (1/2)^{k-2} \kappa(\tilde{r}) d^1
\leq 2 \ca \cc (\norm{\epsilon}_{\Sigma^{-1}_0} + \beta)^2.
\end{equation}

It remains to argue that \(\widehat{\bld{m}}\) is the \emph{unique} stationary point of~\eqref{eq:estproblempoints_fixedN} on~${B_{W_\dagger}}(\bld{m}^{\dagger},(3/2)\hat{r})$.
Replacing \(\hat{r}\) with \(\tilde{r} = (3/2)\hat{r})\) we still obtain the Lipschitz constant~$\kappa((3/2)\hat{r})\leq 3/4$ on the slightly larger ball. Now, assume that~$\widetilde{\bld{m}}$ is any stationary point in the larger ball, thus also fixed point of~$T$ and
\begin{align*}
\norm{\widetilde{\bld{m}} - \widehat{\bld{m}}}_{W_\dagger} =
\norm{T(\widetilde{\bld{m}}) - T(\widehat{\bld{m}})}_{W_\dagger} \leq (3/4) \norm{\widetilde{\bld{m}} - \widehat{\bld{m}}}_{W_\dagger},
\end{align*}
yielding $\widetilde{\bld{m}} = \widehat{\bld{m}}$. 
\end{proof}

\begin{remark}
\label{rem:constant_C1}
Following \eqref{eq:constant_epsilon} and \eqref{eq:constant_c2}, the constant $\Ca$ in the statement of Proposition~\ref{prop:pertub} can be chosen explicitly as
\begin{align*}
\Ca = \max\left\{ \ca, \frac{4\ca} {r^{\dagger}}, 2\cc \right\}. 
\qquad\qedhere
\end{align*}
\end{remark}

Next, in order to prove Proposition~\ref{prop:admissibility}, we require the following estimates on~$K^*$.
\begin{lemma}\label{lem:estimate_adjoint} Suppose that $\eta(y) = [K^* \Sigma_0^{-1} \zeta] (y)$ for $y \in \Omega_s$, $\zeta \in \R^{N_o}$.  Then
\begin{align*}
\sup_{y \in \Omega_s} |D \eta(y)| \le C_D \norm{\Sigma^{-1/2}_0 \zeta}_2.
\end{align*}
for \(D \in \{\,\operatorname{Id}, \nabla, \nabla^2, \nabla^3\,\}\) and \(C_D \in \{\,C_k, C_k', C_k'', C_k'''\,\}\), respectively.
\end{lemma}
\begin{proof}One can see that $\eta$ can be written as
\begin{align}\label{eq:est_adjoint_1}
\eta(y) = [K^* \Prec \zeta ](y) = \inner{ \Sigma_0^{-1} \zeta, k[\bld{x},y]}_2 = \sum_{n=1}^{N_o} (\Sigma_0^{-1} \zeta)_n k(x_n,y).
\end{align}
Hence, for every $y \in \Omega_s$ there holds
\[
|\eta(y)| \le \sum_{n=1}^{N_o} |(\Sigma_0^{-1} \zeta)_n| |k(x_n,y)| \le \sup_{x \in \Omega_o, y \in \Omega_s} |k(x, y)|\sum_{n=1}^{N_o} |(\Sigma_0^{-1} \zeta)_n| = C_k \norm{\Prec \zeta}_1.
\]
Since $\sum_{n=1}^{N_o} \sigma_{0,n}^{-2} = 1$, we have
\begin{equation*}
\norm{\Sigma_0^{-1/2}v}_1 = \sum_{n=1}^{N_o} \sigma_{0,n}^{-1}v_n \le \sqrt{\sum_{n=1}^{N_o} \sigma_{0,n}^{-2}} \sqrt{\sum_{n=1}^{N_o} |v_n|^2} = \norm{v}_2,\quad \forall v \in \R^{N_o}
\end{equation*}
Hence,
\(
\norm{\Sigma_0^{-1}\zeta}_1 \le \norm{\Sigma^{-1/2}_0 \zeta}_2.
\)
From~\eqref{eq:est_adjoint_1} the other estimates follow similarly by taking derivatives.
\end{proof}
\begin{proof}[Proof of Proposition~\ref{prop:admissibility}]

By the definition of $\widehat{\eta}$, one has
\[
\widehat{\eta}(\widehat{y}_n) = \sign(\widehat{q}_n) = \sign (q_n^{\dagger}) \text{ and }\nabla \widehat{\eta}(\widehat{y}_n) = 0,\quad n = 1,\ldots, N_s.
\]
We now prove the \(\theta/2\)-admissibility of $\widehat{\eta}$ if \eqref{eq:set_A1_tilde}--\eqref{eq:set_A2} hold, namely
\begin{align}
&\label{eq:certificate_21} -\sign{\widehat{\eta}}(\widehat{{y}}_n) \nabla^2 \widehat{\eta}(\widehat{{y}}_n) \ge \theta |w_n^{\dagger}|^2 \operatorname{Id}, \quad \forall n  = 1,\ldots, N_s, \\
&\label{eq:certificate_22}\abs{\widehat{\eta}(y)} \leq 1-(\theta/2)^2,\quad \forall y \in \Omega_s \setminus \bigcup_{n=1,\ldots, N_s} B_{w_n^{\dagger}}(\widehat{y}_n, \sqrt{\theta/2})
\end{align}
Compare this to~\eqref{eq:coercive_Hess_eta}--\eqref{eq:upper_bd_ball_eta} for \(\eta_{\text{PC}}\). To this end, consider the noisy pre-certificate
\begin{equation}\label{eq:est_eta_0}
\begin{aligned}
\eta_{\text{PC},\epsilon}
&:= - \beta^{-1}\Ke^* \Sigma_0^{-1} (G'(\bld{m}^\dagger)\delta \widehat{\bld{m}} - \epsilon) \\
&= \beta^{-1}\Ke^* \Prec [G'(\bld{m}^\dagger)\mathcal{I}_0^{-1}
(\beta \vectorsgn - G'(\bld{m}^\dagger)^\top \Sigma_0^{-1} \epsilon) + \epsilon] \\
&= \eta_{\text{PC}} - \beta^{-1}\Ke^* \Sigma^{-1/2}_0[\Sigma^{-1/2}_0 G'(\bld{m}^\dagger)\mathcal{I}_0^{-1}
G'(\bld{m}^\dagger)^\top \Sigma_0^{-1/2} - \operatorname{Id}] (\Sigma_0^{-1/2}\epsilon) \\
&= \eta_{\text{PC}} - \beta^{-1}\Ke^* \Sigma^{-1/2}_0 [P - \operatorname{Id}] (\Sigma_0^{-1/2}\epsilon),
\end{aligned}
\end{equation}
where $\eta_{\text{PC}}$ is given in \eqref{eq:precertificate} and \(P\) is an orthogonal projection to the \(N_s - N_o(1+d)\) dimensional range of \(\Sigma_0^{-1/2} G'(\bld{m}^\dagger)\). This implies
\begin{align*}
\widehat{\eta} &= - \beta^{-1}\Ke^* \Prec (G(\widehat{\bld{m}}) - G(\bld{m}^\dagger) - \epsilon) \\
& =\eta_{\text{PC},\epsilon} - \beta^{-1}K^* \Sigma_0^{-1} \left(G(\widehat{\bld{m}}) - G(\bld{m}^\dagger) - G'(\bld{m}^{\dagger})\delta \widehat{\bld{m}}\right) \\
& = \eta_{\text{PC}} - \beta^{-1}\Big[\underbrace{\Ke^* \Sigma^{-1/2}_0 [P - \operatorname{Id}] (\Sigma^{-1/2}_0\epsilon)}_{e_1}
- \underbrace{K^* \Sigma_0^{-1} \left(G(\widehat{\bld{m}}) - G(\bld{m}^\dagger) - G'(\bld{m}^{\dagger})\delta \widehat{\bld{m}}\right)}_{e_2}\Big].
\end{align*}
Applying Lemma \ref{lem:estimate_adjoint}, we have
\begin{align*}
\norm{e_1}_{\mathcal{C}(\Omega_s)} &\le C_k\norm{\Sigma^{-1/2}_0 [P - \operatorname{Id}] \Sigma^{-1/2}_0\epsilon}_1 \\
&\le C_k\norm{[P - \operatorname{Id}] \Sigma^{-1/2}_0\epsilon}_2 \le 
 C_k\norm{\epsilon}_{\Sigma_0^{-1}}.
\end{align*}
In order to estimate $e_2$, we apply Lemma~\ref{lem:estimate_adjoint} and Proposition~\ref{prop:estimate_G} to have
\begin{equation}\label{eq:est_eta_1}
\begin{aligned}
\norm{e_2}_{\mathcal{C}(\Omega_s)} 
&\le C_k \norm{\Sigma_0^{-1} \left(G(\widehat{\bld{m}}) - G(\bld{m}^\dagger) - G'(\bld{m}^\dagger)\delta \widehat{\bld{m}}\right)}_1 \\
&\le C_k \norm{\Sigma_0^{-1/2}\left(G(\widehat{\bld{m}}) - G(\bld{m}^\dagger) - G'(\bld{m}^\dagger)\delta \widehat{\bld{m}}\right)}_2.
\end{aligned}
\end{equation}
Notice that
\begin{align*}
G(\widehat{\bld{m}}) - G(\bld{m}^\dagger) - G'(\bld{m}^\dagger)\delta \widehat{\bld{m}} 
&= G'(\bld{m}^{\dagger})(\widehat{\bld{m}} - \bld{m}^{\dagger} - \delta \widehat{\bld{m}}) \\ &+ \int_0^1 (G'(\bld{m}_{\tau}) - G'(\bld{m}^{\dagger}))(\widehat{\bld{m}} - \bld{m}^{\dagger})\de \tau 
\end{align*}
where $\bld{m}_{\tau} = \bld{m}^{\dagger} + \tau(\widehat{\bld{m}} - \bld{m}^{\dagger})$. Using this together with Propositions \ref{prop:pertub_estimate} and \ref{prop:pertub}, we have
\begin{equation}\label{eq:est_eta_2}
\begin{aligned}
\norm{e_2}_{\mathcal{C}(\Omega_s)} 
&\le C_k \Big(\norm{\Sigma_0^{-1/2}G'(\bld{m}^{\dagger})(\widehat{\bld{m}} - \bld{m}^\dagger - \delta \widehat{\bld{m}})}_2 \\ 
& \hspace{1cm}+ \int_0^1 \norm{\Sigma_0^{-1/2} (G'(\bld{m}_{\tau}) - G'(\bld{m}^{\dagger}))(\widehat{\bld{m}} - \bld{m}^{\dagger})}_2 \de \tau\Big)\\
&\le {C_k} \left({L_G} \norm{\widehat{\bld{m}} - \bld{m}^\dagger - \delta \widehat{\bld{m}}}_{W_\dagger} 
   + {L_{G'}} \norm{\widehat{\bld{m}} - \bld{m}^\dagger }^2_{W_\dagger} \right)\\
&\le {C_k(L_G + L_{G'})} \Ca^2 \left( \norm{\epsilon}_{\Sigma_0^{-1}} + \beta\right)^2.
\end{aligned}
\end{equation}
Combining \eqref{eq:est_eta_0}--\eqref{eq:est_eta_2}, we have
\begin{equation*}
\norm{\widehat{\eta} - \eta_{\text{PC}}}_{\mathcal{C}(\Omega_s)}
\le   {\cd} \beta^{-1} \left[ (\norm{\epsilon}_{\Sigma_0^{-1}} + \beta)^2 + \norm{\epsilon}_{\Sigma^{-1}_0}\right],
\end{equation*}
where ${\cd} := {C_k((L_G + L_{G'}) \Ca^2+1})$. This yields
\begin{equation}\label{eq:certificate_2}
\begin{aligned}
 \abs{\widehat{\eta}(y)} &\le |\widehat{\eta}(y) - \eta_{\text{PC}}(y)| + |\eta_{\text{PC}}(y)| \\
 &\le \norm{\widehat{\eta} - \eta_{\text{PC}}}_{\mathcal{C}(\Omega_s)} + |\eta_{\text{PC}}(y)| \\
 &\le {\cd} \beta^{-1} \left[ (\norm{\epsilon}_{\Sigma^{-1}_0} + \beta)^2 + \norm{\epsilon}_{\Sigma^{-1}_0}\right] + |\eta_{\text{PC}}(y)|.
\end{aligned}
\end{equation}

We first prove \eqref{eq:certificate_22}. Assume that {\eqref{eq:set_A1_tilde} holds}. Using Proposition \ref{prop:pertub}, we know that for $y \in \Omega_s \setminus \bigcup_{n=1,\ldots,N_s} B_{w_n^{\dagger}}(\widehat{y}_n, \sqrt{\theta/2})$, there holds 
\[\norm{w_n^{\dagger}(y - y_n^{\dagger})}_2 
\ge \norm{w_n^{\dagger}(y - \widehat{y}_n)}_2 - \norm{\widehat{\bld{m}} - \bld{m}^{\dagger}}_{W_\dagger} 
\ge \sqrt{\theta/2} - \sqrt{\theta/32}
= \sqrt{9\theta/32}.\]
Hence, since $\eta_{\PC}$ is non-degenerate, we have by \eqref{eq:nondegenerate_1} that
\[\abs{\eta_{\text{PC}}(y)} \leq 1 - \theta \min\{\theta, \norm{w_n^{\dagger}(y- y^\dagger_n)}^2_2\} \leq 1 - \theta \min\{\theta, 9\theta/32\} = 1 - 9\theta^2/32.\]
This, \eqref{eq:certificate_2} and condition \eqref{eq:set_A2} with $\Cb = \cd$ imply that
\[
|\widehat{\eta}(y)| \le \theta^2/32 + (1 - 9\theta^2/32)
= 1 - (\theta/2)^2, \quad \text{for every }y \in \Omega_s \setminus \bigcup_{n=1,\ldots,N_s} B_{w_n^{\dagger}}(\widehat{y}_n, \sqrt{\theta/2}),
\]
which is indeed \eqref{eq:certificate_22}.

We next prove \eqref{eq:certificate_21}. Following the same arguments as for \eqref{eq:est_eta_1}--\eqref{eq:est_eta_2} together with Lemma~\ref{lem:estimate_adjoint}, we have for $\cd'' := C''_k((L_G + L_{G'}) \Ca^2+1)$ that
\begin{equation}\label{eq:est_hessian_3}
\sup_{y \in \Omega_s} \norm{\nabla^2 \widehat{\eta} - \nabla^2 \eta_{\text{PC}}}_{2 \to 2} 
\le  \cd'' \beta^{-1}\left[ (\norm{\epsilon}_{\Sigma^{-1}_0} + \beta)^2 + \norm{\epsilon}_{\Sigma^{-1}_0}\right].
\end{equation}
In addition, by invoking Assumption~\ref{ass:kernel} on the boundedness of the third derivative of $k$, we obtain
\begin{equation}\label{eq:est_hessian_4}
\begin{aligned}
\norm{\nabla^2 \eta_{\text{PC}}(\widehat{y}_n) - \nabla^2 \eta_{\text{PC}}(y_n^{\dagger})}_{2 \to 2} &\le \norm{\widehat{y}_n - y_n^{\dagger}}_2 \sup_{y \in \Omega_s} \norm{\nabla^3 \eta_{\text{PC}}(y)}_{2 \times 2 \to 2} \\
&\le \abs{w^\dagger_n}^{-1} \norm{w^\dagger_n(\widehat{y}_n - y_n^{\dagger})}_2\,C_k''' \norm{\Sigma_0^{-1/2} G'(\bld{m}^{\dagger})\mathcal{I}_0^{-1}\vectorsgn}_2 \\
&\le \abs{w^\dagger_n}^{-1} \norm{\widehat{\bld{m}} - \bld{m}^{\dagger}}_{W_\dagger} C_k''' L_G \normI \norm{\vectorsgn}_{W_\dagger}  \\
&\le \abs{w^\dagger_n}^{-1} C_k''' L_G \sqrt{\norm{\bld{q}^\dagger}_1}  \normI \Ca(\norm{\epsilon}_{\Sigma^{-1}_0} + \beta)\\
&\le \ce  \beta^{-1} \left[ (\norm{\epsilon}_{\Sigma_0^{-1}} + \beta)^2 + \norm{\epsilon}_{\Sigma^{-1}_0}\right],
\end{aligned}
\end{equation}
with \(\ce := \max_n\abs{w^\dagger_n}^{-1} C_k''' L_G \sqrt{\norm{\bld{q}^\dagger}_1} \normI \Ca\).

From \eqref{eq:est_hessian_3}--\eqref{eq:est_hessian_4}, we have
\begin{align*}
\norm{\nabla^2 \widehat{\eta}(\widehat{y}_n) - \nabla^2 \eta_{\text{PC}}(y_n^{\dagger})}_{2 \to 2}
&= \norm{\nabla^2 \widehat{\eta}(\widehat{y}_n) - \nabla^2 \eta_{\text{PC}}(\widehat{y}_n)}_{2 \to 2} + \norm{\nabla^2 \eta_{\text{PC}}(\widehat{y}_n) - \nabla^2 \eta_{\text{PC}}(y_n^{\dagger})}_{2 \to 2} \\
& \le (\cd'' + \ce) \beta^{-1}\left[ (\norm{\epsilon}_{\Sigma^{-1}_0} + \beta)^2 + \norm{\epsilon}_{\Sigma^{-1}_0}\right]
\end{align*}
for every $n = 1,\ldots, N_s$.  If we set \(\Cb = (\cd'' + \ce)\max_n \abs{w^\dagger_n}^{-2}\) and require
\begin{equation*}
\Cb \beta^{-1} \left[ (\norm{\epsilon}_{\Sigma^{-1}_0} + \beta)^2 + \norm{\epsilon}_{\Sigma^{-1}_0}\right] \le \theta^2/32
\end{equation*}
and that $\eta_{\PC}$ is $\theta-$admissible with \(0 < \theta \leq 1\), we have 
with~\eqref{eq:coercive_Hess_eta} for \(\eta_{\text{PC}}\) for any vector \(\xi\) that
\begin{align*}
 -\sign{\widehat{\eta}}(\widehat{{y}}_n) \xi^\top \nabla^2 \widehat{\eta}(\widehat{{y}}_n) \xi
 &\geq  -\sign{\eta_{\text{PC}}}({y}^\dagger_n) \xi^\top\nabla^2 \eta_{\text{PC}}({y}^\dagger_n) \xi
 - \norm{\xi}^2_2\norm{\nabla^2 \widehat{\eta}(\widehat{y}_n) - \nabla^2 \eta_{\text{PC}}(y_n^{\dagger})}_{2 \to 2}\\
 &\geq 2\theta |w_n^{\dagger}|^2 \norm{\xi}^2_2 - \theta^2/32 |w_n^{\dagger}|^2 \norm{\xi}^2_2
 \geq \theta |w_n^{\dagger}|^2 \norm{\xi}^2_2
\end{align*} and
$\widehat{\eta}$ satisfies \eqref{eq:certificate_21}. Hence, we conclude that $\widehat{\eta}$ is $\theta/2$-admissible for $\widehat{\mu}$.
\end{proof}

\begin{remark}
\label{rem:constant_C2}
In fact, the constant $\Cb$ in the proof of Proposition~\ref{prop:admissibility} can be chosen as
\[
\Cb = \max \left\{ \cd, (\cd'' + \ce)\max_n \abs{w^\dagger_n}^{-2}\right\}. 
\]
Since these constants depend monotonically on $\normI$, we also have the monotone dependence of $\Cb$ on $\normI$.
\end{remark}

\section{Discussion on possible distance candidates}
\label{app:distance_metric}
The Hellinger-Kantorovich distance introduced in this paper turns out to be a suitable distance to quantify the reconstruction. Nevertheless, it is worth mentioning that other choices of distances are possible, for instance the Kantorovich-Rubinstein distance.

\subsection*{Kantorovich-Rubinstein distance} The distance induced by the Kantorovich-Rubinstein (KR) norm  (equivalent to the ``Bounded-Lipschitz'' norm) is also referred to as the ``flat'' metric, and metricises weak* convergence on bounded sets in $\mathcal{M}(\Omega_s)$. The norm can be defined by
\[
\norm{\miu}_{\KR} := \sup \left\{ \int_{\Omega_s} f \de \miu: f \in C^{0,1}(\Omega_s), \norm{f}_{C(\Omega_s)} \le 1 \text{ and } \operatorname{Lip} (f) \le 1 \right\}. \]
Here, $C^{0,1}(\Omega_s)$ is the space of Lipschitz functions on $\Omega_s$ and \[
\operatorname{Lip}(f) := \sup_{x \neq y} \dfrac{\norm{f(x) - f(y)}}{\norm{x-y}}, \quad x, y \in \Omega_s.
\]
The KR distance is then set to
\[
d_{\KR}(\miu_1,\miu_2) = \norm{\miu_1 - \miu_2}_{\KR}.
\]
Although its evaluation for general sparse measures requires the solution of a minimization problem (see~\cite{lellmann_lorenz_schonlieb_valkonen_2014}) it has many useful properties.
For positive measures \(\miu_1,\miu_2 \geq 0\) it is equivalent to a generalized Wasserstein-\(1\) distance for measures not necessarily of the same TV-norm as in~\cite{PiccoliRossi:2016};
\begin{align}
\label{eq:KR_Wasserstein}
d_{\KR}(\miu_1,\miu_2) 
= \inf_{\{\,\tilde{\miu}_1,\tilde{\miu}_2 \colon \; \norm{\tilde{\miu}_1}_{\M} =\norm{\tilde{\miu}_2}_{\M} \,\}} 
\left[
\norm{\tilde{\miu}_1 - \miu_1}_{\M} + \norm{\tilde{\miu}_2 - \miu_2}_{\M} + W_1(\tilde{\miu}_1,\tilde{\miu}_2)
\right].
\end{align}
For signed measures, we use the Jordan decomposition \(\miu_i = \miu^+_i - \miu^-_i\) and observe that
\begin{align*}
d_{\KR}(\miu_1, \miu_2) = \norm*{(\miu_1^+ + \miu_2^-) - (\miu_2^+ + \miu_1^-)}_{\KR}
= d_{\KR} (\miu_1^+ + \miu_2^-, \miu_2^+ + \miu_1^-),
\end{align*}
which then allows to apply the characterization~\eqref{eq:KR_Wasserstein}.
This representation, together with the help of \cite[Proposition~2]{PiccoliRossi:2016} allows to characterize the KR distance of single point sources with weight of equal sign:
\begin{align}
\notag
d_{\KR}(q_1\delta_{y_1},q_2\delta_{y_2})
&=
\min_{0 \leq \theta \sign{q_1} \leq  \min \{ |q_1|, |q_2| \}}\left[
|\theta - q_1| + |\theta - q_2| + \theta \norm{y_1 - y_2}_2
\right] \\
\label{eq:KR_distance}
&=
|q_1 - q_2| + \min \{ |q_1|, |q_2| \} \min\{\norm{y_1 - y_2}_2, 2\}.
\end{align}
For the case of \(\sign q_1 \neq \sign q_2\), we instead have \(d_{\KR}(q_1\delta_{y_1},q_2\delta_{y_2}) = |q_1| + |q_2|\).
The above formula can be used for all finitely supported measures with the same number of support points by applying the triangle inequality.
Motivated by property~\eqref{eq:KR_distance}, the Kantorovich-Rubinstein distance \(d_{\KR}\) is also appropriate to quantify the distance between two discrete measures, and a similar upper bound as in Proposition~\ref{prop:HK_distance_estimate} can be obtained, i.e.,
\[
d_{\KR}(\miu,\miu^\dagger)
\leq 
\sum_{n=1}^{N}\left(
  \abs{q_n - q_n^\dagger}
+ \abs{q_n^\dagger} \, \norm{y_n - y_n^\dagger}_2
\right)
\leq
2\sqrt{2 \norm{\miu^\dagger}} \norm{\bld{m} - \bld{m}^\dagger}_{W_\dagger},
\]
However, this bound either uses an~\(\ell_1\) like sum over the weighted errors, which is not suitable for the rest of our analysis, or the equivalence between a weighted \(\ell_1\) and \(\ell_2\) norm (by H\"older's inequality), and is not an asymptotically sharp bound.
\section{Notation table} \label{appendixnotation}
\setlength\extrarowheight{3pt}
\begin{longtable}{lp{12cm}}	
$\Omega_s $,~$\Omega_o $ & location and observation set, both compact  \\
$N^\dagger_s $,~$y^\dagger_n $,~$q^\dagger_n $ & Unknown number, positions and coefficients of ground truth sources\\
$\bld{y}^\dagger,~\bld{q}^\dagger,~\bld{m}^\dagger $ & Concatenated source/measurement locations, coefficients,~$\bld{m}^\dagger=(\bld{y}^\dagger;\bld{q}^\dagger)$ \\
$w^\dagger$,~$W^\dagger$ & weight vector and weight matrix induced by~$\bld{q}^\dagger$,~\eqref{eq:weighted_norm} et sqq. \\
$N_o $, $x_j $ & Given number and locations of measurements \\
$k $ & Integral kernel, see Section~\ref{subsec:kernels}\\
$\nabla_y k $,~$\nabla^2_{yy} k$,~$\nabla^3_{yyy} k$ &  (Higher-order) partial derivatives of~$k$ w.r.t~$y$ \\
$C_k, C'_k,  C''_k, C'''_k$ &  Bounds on~$k$,~$\nabla_y k $,~$\nabla^2_{yy} k$,~$\nabla^3_{yyy} k$, see Assumption~\ref{ass:kernel} \\
$k[\bld{x},y]$,~$k[x,\bld{y}]$ & Evaluation of~$k(\cdot,y)$, $k(x,\cdot)$ along $\bld{x}$, $\bld{y}$, respectively,~\eqref{eq:evalkx},~\eqref{eq:evalky} \\
$\nabla_y^\top k[x,\bld{y}]$ & Evaluation of $\nabla_y k(x,\cdot)^\top$ along~$\bld{y}$,~\eqref{eq:evalnabky} \\
$k[\bld{x},\bld{y}]$, $\nabla_y^\top k[\bld{x},\bld{y}]$ & Evaluation of $k[\cdot,\bld{y}]$, $\nabla_y^\top k[\cdot,\bld{y}]$ along~$\bld{x}$,~\eqref{eq:evalkxy},~\eqref{eq:nabkxy}\\
$K ,~K^*$ & Source-to-measurements operator,~\eqref{eq:sourcetomeas}. (pre)-adjoint~$K=(K^*)^*$,~\eqref{eq:preadjoint} \\
$\varepsilon$ & Measurement noise, deterministic or random\\
$z^d(\varepsilon) $ & Observed measurements given noise~$\varepsilon$,~\eqref{eq:defmeasurements} \\ 
$\mathcal{M}(\Omega_s)$,~$\norm{\cdot}_{\M(\Omega_s)}$ & Space of Radon measures on~$\Omega_s$ and associated norm, Section~\ref{subsec:Radon}
\\
$\miu^\dagger$ & Sparse ground truth measure,~\eqref{def:sparse} \\
\eqref{def:minprob}, ~\eqref{eq:estproblemmeasure0_bis2} & Minimum norm problem, regularized problem \\ 
$\mathfrak{M}(\eps)$ & Solution set of Problem~\eqref{eq:estproblemmeasure0_bis2} \\
$\beta(\varepsilon),~\beta(p),~\beta_0$ & Parameter choice rules,~$\beta(p)= \beta_0 /\sqrt{p}$  \\
$\eta^\dagger,~\eta_{\PC}, ~\bar \eta $ & Minimum norm dual certificate,~\eqref{eq:defminimalcert}, vanishing derivative pre-certificate,~\eqref{eq:defprecertif} dual certificate, dual certificate of Problem~\eqref{eq:estproblemmeasure0_bis2}\\
$\theta$ & Non-degeneracy parameter, Definition~\ref{def:nondegeneracy}\\
$d_{\TV},~d_{\KR},~d_{\HK}$ & Total variation, Kantorovich-Rubinstein, Hellinger Kantorovich metrics \\
$\mathcal{I}_0,~\rho$ & Fisher information and sign vector,~\eqref{eq:deffishersign}
\\ $p,~\Sigma_0$ & Overall precision of measurements, normalized covariance matrix
\\ $\Sigma,~\gamma_p$ & Parametrized covariance matrix~$\Sigma=p^{-1} \Sigma_0$ and Gaussian~$\gamma_p= \mathcal{N}(0,\Sigma)$ 
\\ $G(\bld{m})$ & Measurements~$k[\bld{x},\bld{y}] \bld{q}$ given~$\bld{m}=(\bld{y}; \bld{q})$,~\eqref{eq:paramtoobs}\\$\widehat{\bld{m}}(\varepsilon),~\delta \widehat{\bld{m}}(\varepsilon) $ & Stationary point of~\eqref{eq:estproblempoints_fixedN}, linear approximation of~$\widehat{\bld{m}}(\varepsilon)$,~\eqref{eq:delta_z_hat}.

\end{longtable}
\end{document}